\documentclass[12pt]{amsart}

\usepackage{amssymb}
\usepackage{bbm}
\usepackage{amsthm}
\usepackage{dsfont}
\usepackage{amsmath}
\usepackage{enumerate}
\usepackage{graphicx}

\newtheorem{proposition}{Proposition}

\newtheorem{lemma}[proposition]{Lemma}

\newtheorem{theorem}[proposition]{Theorem}

\newcommand{\low}[1]{{\underline{#1}}}
\newcommand{\up}[1]{{\overline{#1}}}
\def\lpr{\underline P}
\def\upr{\overline P}
\newcommand{\RR}{\mathds{R}}

\newcommand{\marphi}{\varphi}
\newcommand{\marpsi}{\psi}
\newcommand{\marcop}{C^{\mathrm{M}}}
\newcommand{\marsetcop}{\mathcal C^{\mathrm{M}}}

\newcommand{\mmphi}{\varphi}
\newcommand{\mmpsi}{\chi}
\newcommand{\mmcop}{C^{\mathrm{MM}}}
\newcommand{\mmsetcop}{\mathcal C^{\mathrm{MM}}}

\newcommand{\rmmcop}{C^{\mathrm{RMM}}}
\newcommand{\rmmsetcop}{\mathcal C^{\mathrm{RMM}}}

\begin{document}

\title{Some multivariate imprecise shock model copulas}

\author[D.\ Dol\v{z}an]{David Dol\v{z}an}
\address{David Dol\v{z}an, Faculty of Mathematics and Physics, University of Ljubljana, and Institute of Mathematics, Physics and Mechanics, Ljubljana, Slovenia}
\email{david.dolzan@fmf.uni-lj.si}

\author[D.\ Kokol Bukov\v{s}ek]{Damjana Kokol Bukov\v{s}ek}
\address{Damjana Kokol Bukov\v{s}ek, School of Economics and Business, University of Ljubljana, and Institute of Mathematics, Physics and Mechanics, Ljubljana, Slovenia}
\email{damjana.kokol.bukovsek@ef.uni-lj.si}

\author[M.\ Omladi\v{c}]{Matja\v{z} Omladi\v{c}}
\address{Matja\v{z} Omladi\v{c}, Institute of Mathematics, Physics and Mechanics, Ljubljana, Slovenia}
\email{matjaz@omladic.net}

\author[D.\ \v{S}kulj]{Damjan \v{S}kulj}
\address{Damjan \v{S}kulj, Faculty of Social Sciences, University of Ljubljana, Slovenia}
\email{damjan.skulj@fdv.uni-lj.si}

\thanks{Damjan \v{S}kulj acknowledges the financial support from the Slovenian Research Agency (research core funding No. P5-0168).
David Dol\v{z}an, Damjana Kokol Bukov\v{s}ek, and Matja\v{z} Omladi\v{c} acknowledge financial support from the Slovenian Research Agency (research core funding No. P1-0222).}

\maketitle

\begin{center}
\textbf{Abstract}
\end{center}

\noindent {\small  Bivariate imprecise copulas have recently attracted substantial attention
. However, the multivariate case seems still to be a ``blank slate''. It is then natural that this idea be tested first on shock model induced copulas, a family which might be the most useful in various applications. We investigate a model in which some of the shocks are assumed imprecise and develop the corresponding set of copulas. In the Marshall's case we get a coherent set of distributions and a coherent set of copulas, where the bounds are naturally corresponding to each other. The situation with the other two groups of multivariate imprecise shock model induced copulas, i.e., the maxmin and the the reflected maxmin (RMM) copulas, is substantially more involved, but we are still able to produce their properties. These are the main results of the paper that serves as the first step into a theory that should develop in this direction. In addition, we unfold the theory of bivariate imprecise RMM copulas that has not yet been done before.}

\section{Introduction}\label{intro}

Copulas arising from shock models in the presence of probabilistic uncertainty, which means that probability distributions are not necessarily precisely known, have been proposed for the first time by Omladi\v{c} and \v{S}kulj \cite{OmSk} in bivariate setting. The main purpose of this paper is to present some extensions of the results presented there including an expansion of these notions to the multivariate case.

Copulas have been introduced in the precise setting by A.~Sklar \cite{Skla}, who considered copulas as functions $C(\mathbf{u})=C(u_1,u_2,\ldots,u_n)$ satisfying certain conditions. They can be defined equivalently as joint distribution functions of random vectors with uniform marginal distributions.
He proved a two-way theorem: firstly, given a random vector $\mathbf{X}=(X_1,X_2, \ldots,X_n)$ with a vector of marginal probability distributions $\mathbf{F}= (F_1,F_2, \ldots,F_n)$ and a copula $C$, the function $C(\mathbf{F}(\mathbf{x})) = C(F_1(x_1),F_2(x_2),\ldots ,F_n(x_n))$ is a joint distribution of the random vector $\mathbf{X}$ having distributions $\mathbf{F}$ as its marginals.
Secondly, given a random vector $\mathbf{X}$ with joint distribution $H(\mathbf{x})$ there exists a copula $C(\mathbf{u})$ such that $H(\mathbf{x})= C(\mathbf{F}(\mathbf{x}))$, where $\mathbf{F}$ is the vector of the marginal distribution functions of the respective random variables $\mathbf{X}$. Since then, copula models have become popular in various applications in view of their ability to describe the relationships among random variables in a flexible way and several families of copulas have been introduced to this end, motivated by specific needs from the scientific practice (cf.~\cite{Joe,Nels,DuSe}).

Among the first widely studied and applied families of copulas were the ones arising in shock models. They appear naturally as models of joint distributions for random variables representing lifetimes of components affected by shocks. Two types of shocks are usually considered in these models, the first type only affects each one of the components separately (the idiosyncratic shocks), while the second one simultaneously affects all the components (the exogenous shock).
In the original Marshall's case (cf.~\cite{Mars} based on an earlier work of Marshall and Olkin \cite{MaOl}) both types of shocks cause the component to cease to work immediately. Recently a new family of shock induced copulas has been proposed by Omladi\v{c} and Ru\v{z}i\'{c} \cite{OmRu} where the exogenous, i.e., systemic, shock has a detrimental effect on some of the components and a beneficial effect on the other ones. A third type of shock model induced copulas was introduced by Ko\v{s}ir and Omladi\v{c} \cite{KoOm}, the reflected maxmin copulas, RMM for short. Actually, the two papers introduce the bivariate version for independent shocks, while an extension to somewhat more general multivariate setting is presented in papers \cite{DuOmOrRu,KoBuKoMoOm2}.

Quantitative modeling of uncertainty is traditionally based on the use of precise probabilities: for each event $A$, a single probability $P(A)$ is assigned, universally accepted to satisfy Kolmogorov's axioms. There have been many successful applications of this concept, but also some criticism. The requirement that $P$ be $\sigma$-additive should be replaced, as some believe, by a more realistic requirement that it be additive.
A more flexible theory of uncertainty that has evolved is the concept of imprecise probabilities. For an event $A$, the lower probability $\underline{P}(A)$ can informally be interpreted as reflecting the evidence certainly in favour of the event $A$, while the upper probability $\overline{P}(A)$ reflects all evidence possibly in favour of $A$.
So, the imprecise probability of $A$ may be seen as the set of values lying between the two extremes. A comprehensive study of this notion started by Walley \cite{Wall}, while more recent development in the area can be found in \cite{AuCoCoTr}. It is natural to assume probabilities in these considerations to be finitely additive and not necessarily $\sigma$-additive. The imprecise distribution of a random variable then consists of the interval of all distributions between a lower bound $\underline{F}$ and an upper bound $\overline{F}$; this set is called a \emph{probability box}, a \emph{$p$-box} for short \cite{FeKrGiMySe,TrDe}.

One can find numerous arguments for imprecision, such as scarcity of available information, costs connected to acquiring precise inputs or even inherent uncertainty related to phenomena under consideration. Ignoring imprecision may lead to deceptive conclusions and consequentially to harmful decisions, especially if the conclusions are backed by seemingly precise outputs. Methods of imprecise probabilities have been applied to various areas of probabilistic modeling, such as stochastic processes \cite{DeCoHeQu, skul}, game theory \cite{MiMo, Nau2011}, reliability theory \cite{Cool, ObKiSc, UtCo, YuDeSaSc}, decision theory \cite{JaScAu, MoMiMo, Troffaes2007}, financial risk theory \cite{PeVi, Vicig2008}, and others. Perhaps the first application of the theory of copulas to models of imprecise probabilities has been proposed by Schmelzer~\cite{Schmelzer2015,Schmelzer2015b, Schmelzer2018}.

A possible definition of an \emph{imprecise bivariate $p$-box} was given in Pelessoni et al.~\cite{PeViMoMi2} thus raising the question of the corresponding Sklar type theorem.
The first move in this direction was made by Montes et al.~\cite{MoMiPeVi} proving one half of the imprecise Sklar's theorem using the definition of bivariate $p$-box introduced in \cite{PeViMoMi2}. The same authors introduce in an earlier paper \cite{PeViMoMi1} an \emph{imprecise copula} as an interval of quasi-copulas satisfying certain axioms.
The four authors propose a coherence question in these papers that is answered in the negative by Omladi\v{c} and Stopar \cite{OmSt1}; continuing their work in \cite{OmSt2}, the same authors give a full scale Sklar's theorem in the bivariate imprecise setting using a slightly different notion than the bivariate $p$-box of \cite{PeViMoMi2}, i.e., what they call a \emph{restricted bivariate $p$-box}. Perhaps an even more important result there is \cite[Theorem 4]{OmSt2} saying that if a joint distribution function emerges on a finitely additive probability space, the resulting copula exists and may be chosen so that it satisfies the usual Sklar's axioms. So, all the possible problems that may arise from relaxing the Kolmogorov's $\sigma$-additivity axiom, stay exclusively in the univariate marginal distributions.

As explained earlier, the main contribution of this paper is on the multivariate level, where we extend to the imprecise setting all the three types of shock model induced copulas: Marshall's, maxmin and RMM. For the bivariate case the Marshall's and the maxmin copulas have been first presented in \cite{OmSk}, while for the RMM copulas this has not been done yet, so we have to do it first. The paper is organized as follows.
In Section 2 we revisit some information on the imprecise distributions and $p$-boxes and in Section 3 we present some details on copulas and shock model induced copulas. Section 4 brings facts on the two known families of bivariate shock model induced copulas, the Marshall's and the maxmin ones.
Section 5 unfolds our first main result -- the imprecise version of the bivariate reflected maxmin copulas. In Sections 6, 7, and 8 we give our multivariate extensions of the imprecise shock model based copulas, namely the Marshall's, the maxmin, respectively the RMM copulas.

\section{Theory of imprecise probabilities revisited}\label{sec:impre}

\subsection{Coherent lower and upper probabilities.}\label{subsec:prob} We first introduce briefly the basic concepts and ideas of imprecise probability models. For a detailed treatment, the reader is referred to \cite{AuCoCoTr, Wall}.
Let $\Omega$ be a possibility space, and $\mathcal A$ a collection of its subsets, called \emph{events}.
Usually we assume $\mathcal A$ to be an algebra, but not necessarily a $\sigma$-algebra.

The concept of precise probability on the measurable space $(\Omega, \mathcal A)$ can be generalised by allowing probabilities of events in $\mathcal A$ to be given in terms of intervals $[\low P(A), \up P(A)]$ rather than precise values.
The functions $\low P\leqslant \up P$ are mapping events to their lower and upper probability bounds and are respectively called \emph{lower} and \emph{upper probabilities}.
If $\mathcal A$ is an algebra
, then the following conjugacy relation between lower and upper probabilities is usually required:
\begin{equation}\label{eq-conjugacy}
\upr (A) = 1-\lpr(A^c) ~\text{for every } A\in \mathcal A.
\end{equation}
To every pair of lower and upper probabilities $\lpr$ and $\upr$ we can also associate the set
\begin{equation*}
\mathcal M = \{ P \colon P \text{ is a finitely additive probability on } \mathcal A, \lpr \leqslant P \leqslant \upr \}.
\end{equation*}
It is clear from the above that the set $\mathcal M$ is non-empty only if $\lpr \leqslant \upr$. 

Another central question regarding a pair of lower and upper probabilities is whether the bounds are pointwise limits of the elements in $\mathcal M$:
\begin{equation*}
\lpr (A) = \inf_{P\in\mathcal M} P(A), \qquad \upr (A) = \sup_{P\in\mathcal M} P(A) \qquad \text{ for every} A\in \mathcal A.
\end{equation*}
If the above conditions are satisfied, $\lpr$ and $\upr$ are said to be \emph{coherent} lower and upper probabilities respectively.
In the case of coherence, the conjugacy Condition \eqref{eq-conjugacy} is automatically fulfilled, which means in particular that if a lower probability $\lpr$ is coherent, then it uniquely determines the corresponding upper probability.
A simple characterization of coherence in terms of the properties of $\lpr$ and $\upr$ does not seem to be known in the literature. 


Instead of the full structure of probability spaces, we are often concerned only with the distribution functions of specific random variables. The set of relevant events where the probabilities have to be given then shrinks considerably. In the precise case, a single distribution function $F$ describes the distribution of a random variable $X$, which gives the probabilities of the events of the form $\{ X\leqslant x\}$. Thus $F(x) = P(X\leqslant x)$ for $x\in\overline{\mathds{R}}$ where $\overline{\mathds{R}} =\mathds{R}\cup\{-\infty,+\infty\}$. Sometimes we will also consider the corresponding \emph{survival function}, which we will denote by $\widehat F(x) = 1-F(x) = P(X > x)$, which is decreasing and positive. (In the copula theory literature it is more usual to denote the survival function of $F$ by $\overline{F}$, but we will reserve this notation for a different meaning.)  Notice that the same operator $\widehat{\cdot}$ sends a survival function back to its distribution function. Observe also that in the standard probability theory distribution functions are cadlag, i.e., continuous from the right, and survival functions are caglad, i.e., continuous from the left, while in the finitely additive approach the only property a distribution function has is monotone increasing, and survival function is only monotone decreasing.

\subsection{Bivariate $p$-boxes.} In the imprecise case, the probabilities of the above form are replaced by the corresponding lower (and upper) probabilities, resulting in sets of distribution functions called $p$-boxes \cite{FeKrGiMySe,TrDe}. A \emph{$p$-box} is a pair $(\low F, \up F)$ of distribution functions with $\low F\leqslant \up F$, where $\low F(x) = \lpr(X\leqslant x)$ and $\up F(x) = \upr(X\leqslant x)$.  To every $p$-box we associate the set of all  distribution functions with the values between the bounds:
\begin{equation*}
\mathcal{F}_{(\low F, \up F)} = \{ F \colon F \text{ is a distribution function}, \low F \leqslant F \leqslant \up F \}.
\end{equation*}
Clearly, $\mathcal F$ is a convex set of distribution functions. Conversely, since supremum and infimum of any set of distribution functions are themselves distribution functions, every set of distribution functions generates a $p$-box containing the original set.

In the theory of imprecise probabilities, precise probability denotes a probability measure that is finitely additive, and not necessarily $\sigma$-additive, as is the case in most models using classical probabilities. 
The general theory of bivariate $p$-boxes is relatively new \cite{MoMiPeVi, PeViMoMi2}. A mapping $F\colon \overline{\RR}\times \overline{\RR} \to [0, 1]$ is called \emph{standardized} if
	\begin{enumerate}[(i)]
		\item \emph{it is componentwise increasing: } $F(x_1, y) \leqslant F(x_2, y)$ and $F(x, y_1) \leqslant F(x, y_2)$ whenever $x_1\leqslant x_2$ and $y_1\leqslant y_2$ and for all $x, y\in \overline{\RR}$;
		\item $F(-\infty, y) = F(x, -\infty) = 0$ for every $x, y\in \overline{\RR}$;
		\item $F(\infty, \infty) = 1$.
	\end{enumerate}	
If in addition,
	\begin{enumerate}[(iv)]
	\item 	$F(x_2, y_2)-F(x_1, y_2)-F(x_2, y_1)+F(x_1, y_1) \geqslant 0$ for every $x_1\leqslant x_2$ and $y_1\leqslant y_2$,
	\end{enumerate}	
then it is called a \emph{bivariate distribution function}.

A pair $(\low F, \up F)$ of standardized functions, where $\low F\leqslant \up F$, is called a \emph{bivariate $p$-box}.

Observe that (1) neither the infimum nor supremum of a set of bivariate distribution functions need to be a bivariate distribution function; (2) the set
\[
	\mathcal{F}_{(\low F, \up F)} = \{ F\colon \overline{\RR} \times \overline{\RR} \to [0, 1], F \text{ is a bivariate distribution function}, \low F \leqslant F \leqslant \up F \}
\]
may be empty in general. If it is not empty and its pointwise infimum and supremum equals $\low F$ and $\up F$ respectively, then this bivariate $p$-box is said to be \emph{coherent}.

\subsection{Independent random variables}
In the case where probability distributions are known imprecisely, several distinct concepts of independence exist, such as \emph{epistemic irrelevance, epistemic independence} and \emph{strong independence} (see e.g. \cite{Couso2000,Couso2010}). However, as long as $p$-boxes are concerned, all these notions result in the \emph{factorization property}, 
(cf.\ \cite{MoMiPeVi}):
\noindent let a pair of $p$-boxes $(\low F_X, \up F_X)$ and $(\low F_Y, \up F_Y)$ correspond to the distributions of random variables $X$ and $Y$. The bivariate $p$-box $(\low F, \up F)$ is \emph{factorizing} if
	\begin{align*}
	\low F(x, y) & = \low F_X(x) \low F_Y(y), \label{eq-factorizing-p-box-l}\\
	\up F(x, y) & = \up F_X(x) \up F_Y(y).
	\end{align*}
Thus a bivariate $p$-box corresponding to the bivariate distribution of a pair of independent random variables is factorizing, regardless of the type of independence.

\section{Copulas and Shock model copulas revisited}\label{sec:cop}

\subsection{Copulas}\label{subsec:cop}

Copulas present a very convenient tool for modeling dependence of random variables free of their marginal distributions -- only when one inserts these distributions into a copula, it becomes a joint distribution. 	A function $C\colon [0, 1]\times [0, 1]\to [0, 1]$ is called a (bivariate) \emph{copula} if it satisfies the following conditions:\
	\begin{enumerate}[(C1)]
		\item $C(u, 0) = C(0, v) = 0$ for every $u, v \in [0, 1]$;
		\item $C(u, 1) = u$ and $C(1, v) = v$ for every $u, v \in [0, 1]$;
		\item $C(u_2, v_2)-C(u_1, v_2)-C(u_2, v_1)+C(u_1, v_1) \geqslant 0$ for every $0\leqslant u_1\leqslant u_2\leqslant 1$ and $0\leqslant v_1\leqslant v_2\leqslant 1$.
	\end{enumerate}
This definition extends easily to the multivariate situation and the following theorem can also be stated in that generality. We give here only the bivariate case for the sake of better intuitive appeal.

\begin{theorem}[Sklar's theorem\cite{Skla}]
	Let $F\colon \overline{\RR}\times \overline{\RR}\to [0, 1]$ be a bivariate distribution function with marginals $F_X$ and $F_Y$. Then there exists a copula $C$ such that
	\begin{equation}\label{eq-copula-sklar}
	F(x, y) = C(F_X(x), F_Y(y)) \text{ for all } (x,y)\in\overline{\mathbb{R}}\times\overline{\mathbb{R}};
	\end{equation}
	and conversely, given any copula $C$ and a pair of distribution functions $F_X$ and $F_Y$, equation \eqref{eq-copula-sklar} defines a bivariate distribution function.
\end{theorem}
It is our goal to show how some important classes of copulas can be extended to the case of imprecise probability models. Our construction will spread to the general multivariate case.

\subsection{Marshall's copulas revisited}\label{subsec:marsall} Copulas of the form
\begin{equation*}
\marcop_{\marphi, \marpsi}(u, v) = \begin{cases} \displaystyle
uv\min\left\{ \frac{\marphi(u)}{u}, \frac{\marpsi(v)}{v} \right\} & \text{if } uv > 0; \\
0 & \text{if } uv = 0,
\end{cases}
\end{equation*}
where
\begin{enumerate}[(P1)]
	\item $\marphi$ and $\marpsi$ are two increasing real valued maps on $[0, 1]$;
	\item $\marphi(0) = \marpsi(0) = 0$ and $\marphi(1) = \marpsi(1) = 1$; 	
	\item $\marphi^*(u) = \dfrac{\marphi(u)}{u}\colon (0, 1]\to [1, \infty]$ and $\marpsi^*(v) = \dfrac{\marpsi(v)}{v}\colon (0, 1]\to [1, \infty]$ are decreasing,
\end{enumerate}
were first introduced in \cite{Mars} and are called (bivariate) \emph{Marshall's copulas}. They were historically the first shock model induced copulas.
There is an alternative way for writing down this definition which is better for generalizing it to more than 2 dimensions, i.e.,
\[
    \marcop_{\marphi, \marpsi}(u, v) = \begin{cases} \displaystyle
\marphi(u)\marpsi(v)\min\left\{ \frac{u}{\marphi(u)}, \frac{v}{\marpsi(v)} \right\} & \text{if } \marphi(u)\marpsi(v) > 0; \\
0 & \mbox{otherwise}.
\end{cases}
\]
Observe that this definition is equivalent to the previous one as may be seen via a straightforward consideration.
Here is the stochastic interpretation of these copulas and their generating functions $\marphi$ and $\marpsi$ emerging from \cite{Mars}.
\begin{proposition}\label{prop-marshall-properties}
	Let $X, Y, Z$ be independent random variables with corresponding distribution functions $F_X, F_Y$ and $F_Z$. Define $U= \max\{ X, Z\}$ and $V=\max\{ Y, Z\}$ and let $F_U$ and $F_V$ denote their respective distribution functions. Furthermore, let $H$ be the bivariate joint distribution function of the pair $(U, V)$. Then:
	\begin{enumerate}[(i)]
		\item $F_U = F_XF_Z$ and $F_V = F_YF_Z$.
		\item A pair of functions $\marphi$ and $\marpsi$ satisfying (P1)--(P3) exists, so that $F_X(x) = \marphi(F_U(x))$ for all $x$, where $F_U(x)>0$, and $F_Y(y) = \marpsi(F_V(y))$ for all $y$, where $F_V(y)>0$.
		\item $H(x, y) = \marcop_{{\marphi}, {\marpsi}}(F_U(x), F_V(y))$.
		\item $\marphi^*\circ F_U = \marpsi^*\circ F_V$.
		\item $F_Z = \dfrac{F_U}{F_X} = \dfrac{F_U}{\marphi(F_U)} = \dfrac{F_V}{F_Y} = \dfrac{F_V}{\marpsi(F_V)}$, where the expressions are defined.
	\end{enumerate}
\end{proposition}

It turns out that the generating functions are necessarily continuous on the interval $(0,1]$ (cf.\ \cite{OmRu}); however, they are not uniquely determined with the condition \emph{(ii)}.

These copulas were extended to the imprecise setting in \cite{OmSk} for the bivariate case and we will extend this further to the multivariate case.

\subsection{Maxmin copulas revisited}\label{subsec:maxmin} Another family of shock model induced copulas are the so called maxmin copulas introduced recently by Omladi\v{c} and Ru\v{z}i\'{c} \cite{OmRu}. A maxmin copula depends on two functions $\mmphi$ and $\mmpsi\colon [0, 1]\to [0, 1]$, satisfying the properties:
\begin{enumerate}[(F1)]
	\item $\mmphi(0) = \mmpsi(0) = 0$ and $\mmphi(1) = \mmpsi(1) = 1$;
	\item $\mmphi$ and $\mmpsi$ are increasing;
	\item $\mmphi^*(u) = \dfrac{\mmphi(u)}{u}\colon (0, 1]\to [1, \infty]$ and $\mmpsi_*(v) \colon [0, 1]\to [1, \infty]$ are decreasing, where $\mmpsi_*(v) = \begin{cases}
                                       \dfrac{1-\mmpsi(v)}{v-\mmpsi(v)} & \mbox{if } v\neq\mmpsi(v); \\
                                       +\infty & \mbox{if } v=\mmpsi(v)\neq1;  \\
                                       1 & \mbox{if } v=1.
                                     \end{cases}$
\end{enumerate}
A \emph{maxmin copula} is a map $\mmcop\colon [0, 1]\times [0, 1] \to [0, 1]$ defined by
\begin{equation*}\label{eq-maxmin-copula-defintion}
\mmcop_{\mmphi, \mmpsi}(u, v) = uv + \min \{ u(1-v), (\mmphi(u)-u)(v-\mmpsi(v)) \}.
\end{equation*}
Here is the stochastic interpretation of these copulas and of functions $\mmphi$ and $\mmpsi$. (Observe that the random variable $U$ in the following proposition is the same as in Proposition~\ref{prop-marshall-properties}.)
\begin{proposition}\label{prop-maxmin-properties}
	Let independent random variables $X, Y$ and $Z$ be given with respective distribution functions $F_X, F_Y$ and $F_Z$. Define $U=\max\{ X, Z\}$ and $W = \min\{ Y, Z \}$ and let $F_U, F_W$ denote the distribution functions of $U$ and $W$ respectively. Let $H$ be the joint distribution function of $(U, W)$.
	Then:
	\begin{enumerate}[(i)]
		\item $F_U= F_X F_Z$ and $F_W = F_Y + F_Z - F_YF_Z$.
		\item A pair of functions $\mmphi$ and $\mmpsi$ satisfying (F1)--(F3) exists, so that $F_X(x) = \mmphi(F_U(x))$ for all $x$, where $F_U(x)>0$ and $F_Y(y) = \mmpsi(F_W(y))$ for all $y$, where $F_W(y)<1$.
		\item $\mmphi^*\circ F_U = \mmpsi_*\circ F_W$.
		\item In terms of survival functions instead of distribution functions, the second equation in (i) assumes the following equivalent form $\widehat F_W=\widehat F_Y\widehat F_Z$. 		
		\item $H(x, y) = \mmcop_{\mmphi, \mmpsi}(F_U(x), F_W(y))$.
	\end{enumerate}
\end{proposition}
Observe that, as in the first paragraph after Proposition~\ref{prop-marshall-properties}, functions $\mmphi$ and $\mmpsi$ are necessarily continuous, although not unique. Note also that the roles of generating functions $\marphi$ in Marshall's and maxmin models are equivalent, while the roles of $\marpsi$ and $\mmpsi$ may be seen opposite in some sense.

Marshall and maxmin copulas were also extended to the imprecise setting in \cite{OmSk} for the bivariate case. We extend them further to the multivariate case in Sections~\ref{subsec:imp:marsh} and \ref{subsec:imp:maxmin}.

\subsection{Reflected maxmin copulas revisited}\label{subsec:reflected} Let $X,Y,Z$ be independent variables with probability distribution functions $F_X$, $F_Y$, and $F_Z$ respectively. Define $U=\max\{X,Z\}$ and $W= \min\{Y,Z\}$ as in Subsection \ref{subsec:maxmin}. Recall the definition of a survival function from Subsection \ref{subsec:prob}. So, we have
\[
    F_U=F_XF_Z, \quad \mbox{and}\quad \widehat{F}_W= \widehat{F}_Y \widehat{F}_Z\quad\mbox{or}\quad F_W=\widehat{\widehat{F}_Y \widehat{F}_Z}.
\]
From Subsection \ref{subsec:maxmin} we recall the existence of functions $\varphi$ and $\chi$ such that $\varphi(F_U)=F_X$, and $\chi(F_W)=F_Y$, whenever $F_U>0$ and ${F}_W<1$, so that
\begin{equation}\label{eq:defineRMM}
  \varphi(F_XF_Z)=F_X\quad\mbox{and}\quad\chi\left(\widehat{\widehat{F}_Y \widehat{F}_Z}\right)=F_Y.
\end{equation}
Rewrite functions $\varphi$ and $\chi$ into
\begin{equation}\label{eq:rewrite}
  f(x)=\varphi(x)-x,\quad g(x)=1-x-\chi(1-x)
\end{equation}
to get
\[
    f(F_U)=f(F_XF_Z)=\varphi(F_U)-F_XF_Z=F_X-F_XF_Z=F_X\widehat{F}_Z
\]
if $F_U>0$ and
\[
    g(\widehat{F}_W)=F_W-\chi(F_W)=F_Y+F_Z-F_YF_Z-F_Y = \widehat{F}_YF_Z
\]
if $F_W<1$, which summarizes into
\begin{equation}\label{eq:f_g_product}
  f(F_XF_Z)=F_X\widehat{F}_Z\quad\mbox{and}\quad g(\widehat{F}_Y\widehat{F}_{Z})=\widehat{F}_YF_Z,
\end{equation}
whenever $F_XF_Z>0$ and $\widehat{F}_Y\widehat{F}_Z>0$.

There is a general notion of reflection of the variables in the copula theory corresponding to reflection of the horizontal and vertical bisectors of the unit square (cf.\ \cite[p.~30]{DuSe}). In the paper \cite{KoOm} reflection $y\mapsto 1-y$ that turns a general copula $C(x,y)$ into $y-C(x,1-y)$ is used, together with the replacement of the generators $\varphi$ and $\chi$ with $f$ and $g$ to transform the class of maxmin copulas into the class of what the authors call \emph{reflected maxmin copulas}, RMM for short. They also prove (\cite[Lemmas~1\&2]{KoOm}):

\textbf{Claim.} \emph{Conditions (F1)--(F3) are satisfied for the original generating functions of the maxmin copula if and only if the following conditions are satisfied by functions $f$ and $g$:
\begin{enumerate}[(G1)]
  \item $f(0)=g(0)=0$, $f(1)=g(1)=0$, $f^*(1)=g^*(1)=0$,
  \item the functions $\widehat{f}(u)=f(u)+u$ and $\widehat{g}(w)=g(w)+w$ are increasing on $[0,1]$,
  \item the functions $f^*$ and $g^*$ are decreasing on $(0,1]$.
\end{enumerate}}
Here we use the notation from \cite{KoOm} for the functions
\[
    f^*(u)=\frac{f(u)}{u},\ \ g^*(w)=\frac{g(w)}{w},
\]
for $u,w>0$. In addition we define
\[
    f^*(0)=\left\{
             \begin{array}{ll}
               \lim_{u\downarrow0}\frac{f(u)}{u} & \hbox{if it exists;} \\
               \infty & \hbox{otherwise,}
             \end{array}
           \right.
\]
and similarly for $g^*(0)$. Hence $f,g:[0,1]\rightarrow[0,1]$ and $f^*,g^*:[0,1]\rightarrow[0,\infty]$.
Also, using \cite[Theorem 3]{KoOm}, we know that the copula corresponding to the random vector $(U,W)$ with respect to the distribution function of $U$ and survival function of $W$ is equal to
\begin{equation}\label{eq:rmm}
  \rmmcop_{f,g}(x,y)=\max\{0,xy-f(x)g(y)\}.
\end{equation}
So, clearly, the reflected maxmin copula is the copula obtained from the corresponding (maxmin) copula of the random vector $(U,W)$ after applying the reflection on the second variable, and the functions $f$ and $g$ satisfying Conditions (G1)--(G3) are its generators.

\section{Imprecise shock-model copulas revisited -- Marshall's and maxmin}\label{sec:revisited}

\subsection{Order relations generated by shock models.} The theory of $p$-boxes, univariate and bivariate, is based on the sets of probability distributions that lie between the boundary distributions $\low F$ and $\up F$. In order to transfer the theory of Marshall's copulas from precise distribution functions to the more general case of $p$-boxes, the critical step is to determine, whether the order on the set of distribution functions imposed by $p$-boxes is preserved on the corresponding copulas. As shown in \cite{OmSk}, the order is indeed preserved, yet in different ways for Marshall's and maxmin case.

From Subsections \ref{subsec:marsall} and \ref{subsec:maxmin} recall the triples of independent distribution functions $(F_X, F_Y, F_Z)$ which give rise to the pairs of not necessarily independent functions $(F_U, F_V)$ in the Marshall's case, and $(F_U, F_W)$ in the maxmin case; and then further to pairs of generating functions ${\marphi}, {\marpsi}$ and $\mmphi, \mmpsi$, and to corresponding copulas $\marcop_{{\marphi}, {\marpsi}}$ and $\mmcop_{\mmphi, \mmpsi}$ respectively. We follow \cite{OmSk} to introduce imprecision in these models. We allow $F_X$ and $F_Y$ to be imprecise, while for technical reasons $F_Z$ is assumed precise. Replace $F_X$ and $F_Y$ with $p$-boxes $(\low F_X, \up F_X)$ and $(\low F_Y, \up F_Y)$. So, we consider triples $(F_X, F_Y, F_Z)$ where $\low F_X \le F_X \le \up F_X$ and $\low F_Y \le F_Y \le \up F_Y$, or in $p$-box notation $F_X\in\mathcal{F}_{(\low F_X ,\up F_X)}$ and $F_Y\in\mathcal{F}_{(\low F_Y ,\up F_Y)}$.

We now relate respective copulas $\marcop_{{\marphi}, {\marpsi}}$ and $\mmcop_{\mmphi, \mmpsi}$ to the triple $(F_X, F_Y, F_Z)$ via distribution functions $F_U, F_V$ and $F_W$. Let pairs of generating functions $\marphi, \marpsi$, and $\mmphi,\mmpsi$, all mapping $[0, 1]\to [0, 1]$, be such that $\marphi(F_U) = F_X, \marpsi(F_V) = F_Y$ and $\mmpsi(F_W) = F_Y$ whenever $F_U>0,F_V>0,F_W<1$. If they also satisfy the corresponding Conditions (P1)--(P3) and (F1)--(F3), then we say that the triple $(\marphi,\marpsi,\mmpsi)$ is \emph{associated} to the triple $(F_X, F_Y, F_Z)$. In this case we will also say that any of the generating functions $\marphi, \marpsi$, or $\mmpsi$ is associated to the triple $(F_X, F_Y, F_Z)$. Now, these conditions determine the generating functions only on the images of the corresponding distribution functions and with these functions changing within their $p$-boxes we have to adjust the appropriate extensions so that the required order relations are satisfied. For instance, if $(F_X, F_Y, F_Z)$ and $(F'_X, F'_Y, F_Z)$ are two respective triples with $F'_X \le F_X$  and $F'_Y \le F_Y$, relations $\marphi' \le \marphi, \marpsi' \le \marpsi$ and $\mmpsi' \le \mmpsi$ will not be satisfied necessarily for any pairs of triples of generating functions $(\marphi,\marpsi, \mmpsi)$ and $(\marphi',\marpsi', \mmpsi')$ associated with them. Rather surprisingly, it is possible to find explicit formulas for the extensions that do preserve the order. We present here the solution of this problem from \cite{OmSk}.

Denote by $f(x+)$, respectively $f(x-)$, the right limit, respectively the left limit of a monotone (increasing) function $f$ at $x$; note that the limits exist because $f$ is monotone. For distribution functions $F_X$ and $F_Z$ let $F_U=F_XF_Z$. Choose a $u\in (0, 1)$ and let $x_0$ be any value such that $F_U(x_0-) \leqslant u \leqslant F_U(x_0+)$.
Define
\begin{equation}\label{eq-phi-ext}
\mmphi(u) =
\begin{cases}
0 & \text{if } u = 0; \\
F_X(x_0-) & \text{if } u_-\leqslant u \leqslant u_l; \\
\dfrac{u}{F_Z(x_0)} & \text{if } u_l\leqslant u \leqslant u_u; \\
F_X(x_0+) & \text{if } u_u\leqslant u \leqslant u_+ ; \\
1 & \text{if } u = 1,
\end{cases}
\end{equation}
where
\begin{align*}
u_- & = F_X(x_0-)F_Z(x_0-) = F_U(x_0-), & u_l & = F_X(x_0-)F_Z(x_0), \\
u_+ & = F_X(x_0+)F_Z(x_0+) = F_U(x_0+), & u_u & = F_X(x_0+)F_Z(x_0).
\end{align*}
Furthermore, choose a $v\in (0, 1)$ and let $y_0$ be any value such that $F_W(y_0-) \leqslant v \leqslant F_W(y_0+)$. The extension of $\mmpsi$ at $v$ is defined as follows:
\begin{equation}\label{eq-chi-ext}
\mmpsi(v) =
\begin{cases}
0 & \text{if } v = 0; \\
F_Y(y_0-) & \text{if } v_-\leqslant v \leqslant v_l; \\
\dfrac{v-F_Z(y_0)}{1-F_Z(y_0)} & \text{if } v_l\leqslant v \leqslant v_u; \\
F_Y(y_0+) & \text{if } v_u\leqslant v \leqslant v_+ ; \\
1 & \text{if } v = 1,
\end{cases}
\end{equation}
where
\[
\begin{array}{cc}
  v_- & =F_Y(y_0-) + F_Z(y_0-) -F_Y(y_0-)F_Z(y_0-) = F_W(y_0-), \\
  v_l & =F_Y(y_0-) + F_Z(y_0) -F_Y(y_0-)F_Z(y_0),\quad\quad\quad\ \quad\quad\quad \\
  v_+ & =F_Y(y_0+) + F_Z(y_0+) -F_Y(y_0+)F_Z(y_0+) = F_W(y_0+), \\
  v_u & =F_Y(y_0-) + F_Z(y_0) -F_Y(y_0-)F_Z(y_0).\quad\quad\quad\ \quad\quad\quad
\end{array}
\]
The generating functions  obtained using the extension \eqref{eq-phi-ext} for $\marphi$, and appropriately adjusted for $\marpsi$, and the extension \eqref{eq-chi-ext} for $\mmpsi$, are associated with the triple $(F_X, F_Y, F_Z)$. Moreover, the following lemmas hold (cf.\ \cite{OmSk}).
\begin{lemma}\label{lem-phi-order}
	Let $F'_X\leqslant F_X$ and $F_Z$ be given, and let $F_U=F_XF_Z$ and $F'_U=F'_XF_Z$. Then $\mmphi' \leqslant \mmphi$, where $\mmphi'$ and $\mmphi$ are defined by applying \eqref{eq-phi-ext} to $F'_X$ and $F_X$ respectively.
\end{lemma}
\begin{lemma}\label{lem-chi-order}
	Let $F'_Y\leqslant F_Y$ and $F_Z$ be given, and let $F_W=F_Y + F_Z - F_YF_Z$ and ${F'_W}=F'_Y + F'_Z - F'_YF'_Z$. Then $\mmpsi' \leqslant \mmpsi$, where $\mmpsi'$ and $\mmpsi$ are defined by applying \eqref{eq-chi-ext} to $F'_Y$ and $F_Y$ respectively.
\end{lemma}

\subsection{Imprecise Marshall's copulas and imprecise maxmin copulas}\label{subsec:imp:marmm}
Based on the results in the previous subsection, we can now define the imprecise version of the Marshall's and maxmin copulas.
\label{def-i-mar-cop}
	The family of copulas
	\begin{equation}\label{eq-i-mar-cop}
	\marsetcop = \{ \marcop_{\marphi, \marpsi} \colon \low{\marphi}\leqslant  \marphi \leqslant  \up{\marphi}, \low{\marpsi}\leqslant  \marpsi \leqslant  \up{\marpsi}\},
	\end{equation}
where $\low{\marphi} \leqslant  \up\marphi$ and $\low{\marpsi} \leqslant  \up{\marpsi}$, and all $\marphi$ and $\marpsi$, including the bounds, satisfy Conditions (P1)--(P3), is called an \emph{imprecise Marshall's copula}.
\label{def-i-mm-cop}
	The family of copulas
	\begin{equation}\label{eq-i-mm-cop}
	\mmsetcop = \{ \mmcop_{\mmphi, \mmpsi} \colon \low{\mmphi}\leqslant  \mmphi \leqslant  \up{\mmphi}, \low{\mmpsi}\leqslant  \mmpsi \leqslant  \up{\mmpsi}\},
	\end{equation}
where $\low{\mmphi} \leqslant  \up\mmphi$ and $\low{\mmpsi}\leqslant  \up{\mmpsi}$, and all $\mmphi$ and $\mmpsi$, including the bounds, satisfy Conditions (F1)--(F3),
is called an \emph{imprecise maxmin copula}. 

Here is the stochastic interpretation of the imprecise shock model copulas. Let $X$ and $Y$ be random variables, whose distributions are given in terms of $p$-boxes $(\low F_X, \up F_X)$ and $(\low F_Y, \up F_Y)$, and $Z$ a random variable with a precise distribution function $F_Z$.
To every triple $(F_X, F_Y, F_Z)$ where $F_X \in \mathcal{F}_{(\low F_X, \up F_X)}$ and $F_Y\in \mathcal{F}_{(\low F_Y, \up F_Y)}$, there exists a Marshall's copula $\marcop_{\marphi, \marpsi}$ with generating functions $\marphi,\marpsi$ given by \eqref{eq-phi-ext}, such that $F_U$ and $F_V$ are the respective distribution functions of random variables $U=\max\{ X, Z\}$ and $V=\max\{ Y, Z \}$, and $\marcop_{\marphi, \marpsi}(F_U, F_V)$ is their joint distribution function.  Next, we will denote the minimal generating functions associated to the triple $(\low F_X, \low F_Y, F_Z)$ by $\low \marphi$ and $\low \marpsi$, and the corresponding maximal generating functions associated to the triple $(\up F_X, \up F_Y, F_Z)$ by $\up \marphi$ and $\up \marpsi$. (Note that $\low \marphi,\low \marpsi,\up \marphi,\up \marpsi$ themselves are not necessarily constructed by Equation \eqref{eq-phi-ext}. The existence of these functions was proven in \cite[Proposition 4]{OmSk}.) Moreover, we will denote by $\low F_U$ and $\low F_V$ the infimum of the distribution functions of $U$ and $V$ respectively, and by $\up F_U$ and $\up F_V$ the supremum of the distribution functions of $U$ and $V$ respectively.

Similarly, for the maxmin case, given a triple $(F_X, F_Y, F_Z)$ where $F_X \in \mathcal{F}_{(\low F_X, \up F_X)}$ and $F_Y\in \mathcal{F}_{ (\low F_Y, \up F_Y)}$, there exists a maxmin copula $\mmcop_{\mmphi, \mmpsi}$ with generating functions  $\mmphi$ and $\mmpsi$ given by \eqref{eq-phi-ext} and \eqref{eq-chi-ext}, respectively, such that $F_U$ and $F_W$ are the respective distribution functions of random variables $U = \max\{ X, Z\}$ and $W = \min\{ Y, Z \}$, and $\mmcop_{\mmphi, \mmpsi}(F_U, F_W)$ is their joint distribution function. Next, let $\low \mmphi$ and $\low \mmpsi$ be the minimal generating functions associated to the triple $(\low F_X, \low F_Y, F_Z)$, not necessarily constructed by Equations \eqref{eq-phi-ext}\&\eqref{eq-chi-ext}.
Also, let $\up \mmphi$ and $\up \mmpsi$ be the maximal generating functions associated to the triple $(\up F_X, \up F_Y, F_Z)$, not necessarily constructed by Equations \eqref{eq-phi-ext}\&\eqref{eq-chi-ext}. The existence of these functions was likewise proven in \cite[Proposition 4]{OmSk}. Finally, the supremum and infimum of the distribution functions $F_U,F_W$ will be denoted by $\low F_U, \low F_W$ and $\up F_U, \up F_W$ respectively.

The following theorems describe the properties of the imprecise Marshall's and maxmin copulas. Recall that the definitions of functions $\low\marphi^*, \up\marphi^*,  \low \marpsi^*, \up \marpsi^*$ and $\low \mmpsi_*, \up \mmpsi_*$ are exhibited in Conditions (P3) and (F3), respectively.

\begin{theorem}[{Properties of imprecise Marshall's copulas \cite[Theorem 3]{OmSk}}]  \label{main:marshall}
	In the situation described above we have:
	\begin{enumerate}[(i)]
		\item $\low \marphi \leqslant  \up \marphi$ and $\low{\marpsi}\leqslant  \up{\marpsi}$.
		\item $\marcop_{\low{\marphi}, \low{\marpsi}}\leqslant \marcop_{{\marphi}, {\marpsi}}\leqslant \marcop_{\up{\marphi}, \up{\marpsi}}$, where $ \marcop_{{\marphi}, {\marpsi}}$ is the Marshall's copula corresponding to some triple $(F_X, F_Y, F_Z)$, where $F_X\in \mathcal{F}_{(\low F_X,\up F_X)}$ and $F_Y\in\mathcal{F}_{(\low F_Y ,\up F_Y)}$.
		\item
		\[
		\low F_U = \low F_X F_Z \ \ \ \ \ \low F_V  = \low F_Y F_Z \ \ \ \ \
		\up F_U = \up F_X F_Z \ \ \ \ \  \up F_V  = \up F_Y F_Z.
		\]
		\item
		\begin{align*}
		\low F_X(x) & = \low{\marphi}(\low F_U(x)), \text{ if } \low F_U(x)>0; & \up F_X(x) & = \up{\marphi}(\up F_U(x)), \text{ if } \up F_U(x)>0; \\
		\low F_Y(y) & = \low{\marpsi}(\low F_V(y)), \text{ if } \low F_V(y)>0; & \up F_Y(y) & = \up{\marpsi}(\up F_V(y)), \text{ if } \up F_V(y)>0.
		\end{align*}
		\item $\low\marphi^*\circ\low F_U = \low \marpsi^* \circ\low F_V= \up\marphi^*\circ\up F_U = \up \marpsi^* \circ\up F_V$ if $\low F_U, \low F_V, \up F_U, \up F_V>0$ .
		\item $\low F_U\leqslant \up F_U$ and $\low F_V\leqslant \up F_V$.
		\item The distributions of the random variables $U = \max\{ X, Z \}$ and $V=\max\{ Y, Z\}$ are described with the $p$-boxes $(\low F_U, \up F_U)$ and  $(\low F_V, \up F_V)$  respectively.
		\item $\marcop_{\low{\marphi}, \low{\marpsi}}(\low F_U, \low F_V)\leqslant  \marcop_{\up{\marphi}, \up{\marpsi}}(\up F_U, \up F_V)$.
		\item The joint distribution of $(U, V)$ is described by a bivariate $p$-box
		\begin{equation*}
		(\low H, \up H) = (\marcop_{\low{\marphi}, \low{\marpsi}}(\low F_U, \low F_V), \marcop_{\up{\marphi}, \up{\marpsi}}(\up F_U, \up F_V)).
		\end{equation*}				
	\end{enumerate}	
\end{theorem}
\begin{theorem}[{Properties of imprecise maxmin copulas \cite[Theorem 4]{OmSk}}]\label{main:maxmin}
	In the above situation we have:
	\begin{enumerate}[(i)]
		\item $\low\mmphi\leqslant \up \mmphi$ and $\low \mmpsi\leqslant \up \mmpsi$.
		\item $\mmcop_{\low{\mmphi}, \up{\mmpsi}} \leqslant \mmcop_{\mmphi, \mmpsi} \leqslant \mmcop_{\up{\mmphi}, \low{\mmpsi}}$ where $\mmcop_{{\mmphi}, {\mmpsi}}$ is a maxmin copula corresponding to some triple $(F_X, F_Y, F_Z)$, where $F_X\in \mathcal{F}_{(\low F_X, \up F_X)}$ and $F_Y\in \mathcal{F}_{(\low F_Y, \up F_Y)}$.
		\item
		\begin{align*}
		\low F_U & = \low F_X F_Z, & \low F_W & = \low F_Y + F_Z -\low F_Y F_Z, \\
		\up F_U & = \up F_X F_Z, & \up F_W & = \up F_Y + F_Z -\up F_Y F_Z.
		\end{align*}
		\item
		\begin{align*}
		\low F_X(x) & = \low{\mmphi}(\low F_U(x)), \text{ if } \low F_U(x)>0; & \up F_X(x) & = \up{\mmphi}(\up F_U(x)), \text{ if } \up F_U(x)>0; \\
		\low F_Y(y) & = \low{\mmpsi}(\low F_W(y)), \text{ if } \low F_W(y)<1; & \up F_Y(y) & = \up{\mmpsi}(\up F_W(x)), \text{ if } \up F_W(y)<1.
		\end{align*}
		\item $\low \mmphi^* \circ \low F_U = \low \mmpsi_* \circ \low F_W= \up \mmphi^* \circ \up F_U = \up \mmpsi_* \circ \up F_W$ if $\low F_U, \up F_U>0$ and $\low F_W,\up F_W<1$.
		\item $\low F_U\leqslant \up F_U$ and $\low F_W\leqslant \up F_W$.
		\item The distributions of the random variables $U = \max\{ X, Z \}$ and $W=\min\{ Y, Z\}$ are described with the $p$-boxes $(\low F_U, \up F_U)$ and  $(\low F_W, \up F_W)$  respectively.
		\item $\mmcop_{\low{\mmphi}, \low{\mmpsi}}(\low F_U, \low F_W)\leqslant  \mmcop_{\up{\mmphi}, \up{\mmpsi}}(\up F_U, \up F_W)$.
		\item The joint distribution of $(U, W)$ is described by a bivariate $p$-box
		\begin{equation}\label{eq-maxmin-p-box}
		(\low H, \up H) = (\mmcop_{\low{\mmphi}, \low{\mmpsi}}(\low F_U, \low F_W), \mmcop_{\up{\mmphi}, \up{\mmpsi}}(\up F_U, \up F_W)).
		\end{equation}	
	\end{enumerate}
\end{theorem}

\vskip.5cm

\section{Bivariate imprecise reflected maxmin copulas}\label{sec:reflected-bi}

In this section we present one of our main results, the imprecise version of the bivariate  reflected maxmin copulas. Observe that reflected maxmin copulas were first introduced in \cite{KoOm} as a simplification of the maxmin copulas first introduced in \cite{OmRu}, and they are revisited in Subsection \ref{subsec:reflected}. The imprecise version of the maxmin copulas are introduced in \cite{OmSk} and revisited in Subsection \ref{subsec:imp:marmm}. In case of a conflicting notation of the two sources we prefer to use the notation of Subsection \ref{subsec:reflected}.

We assume that the distribution functions of $X$ and $Y$ are imprecise, i.e.\ they are all obtained via finitely additive measures and live in their $p$-boxes
\[
F_X\in\mathcal{F}_{(\underline{F}_X,\overline{F}_X)}\quad\mbox{and}\quad F_Y\in\mathcal{F}_{(\underline{F}_Y,\overline{F}_Y)}.
\]
However, for technical reasons we assume that the distribution function $F_Z$ is precise. Following Subsection \ref{subsec:reflected} we associate to any triple $(F_X,F_Y,F_Z)$ the corresponding generating functions $\varphi$ and $\chi$ satisfying the Conditions (F1)--(F3). Observe that the existence of the kind of generating functions satisfying the defining Relations \eqref{eq:defineRMM} and such that order relations on the respective generating functions $\varphi$ and $\chi$ are in accordance with order relations on the respective distribution functions $F_X$ and $F_Y$ was presented in Subsection \ref{subsec:imp:marmm} via Equations \eqref{eq-phi-ext} and \eqref{eq-chi-ext}. We now transform these generating functions into functions $f$ and $g$ using Equation \eqref{eq:rewrite}. In this sense functions $f$ and $g$ are defined indirectly via Equations \eqref{eq-phi-ext} and \eqref{eq-chi-ext} and satisfy 
Conditions (G1)--(G3). Moreover, the order is preserved for functions $f$ and reversed for functions $g$. We will say that they are associated to the triple $(F_X,F_Y,F_Z)$.

Based on the results presented in Subsection \ref{subsec:reflected} we can now define the imprecise version of the reflected maxmin copulas. Following also the ideas presented in Subsection \ref{subsec:imp:marmm} we introduce the family of copulas depending on two pairs of functions $\low{f} \leqslant \up f$ and $\low{g}\leqslant \up{g}$ satisfying Conditions (G1)--(G3). We let
	\begin{equation}\label{eq-reflected}
	\rmmsetcop = \{ \rmmcop_{f, g} \colon \low{f}\leqslant  f \leqslant  \up{f}, \low{g}\leqslant  g\leqslant  \up{g}\},
	\end{equation}
where all $f$ and $g$ also satisfy Conditions (G1)--(G3). We call this family an \emph{imprecise reflected maxmin copula}. We can expect a possible stochastic interpretation only if in this definition we let $\low f$ be the infimum of all functions $f$ associated to the triple $(\low F_X,\low F_Y,F_Z)$ and we let $\up f$ be the supremum of all functions $f$ associated to the triple $(\up F_X,\up F_Y,F_Z)$. Furthermore, let $\low g$ be the infimum of all functions $g$ associated to the triple $(\up F_X,\up F_Y,F_Z)$ and let $\up g$ be the supremum of all functions $g$ associated to the triple $(\low F_X,\low F_Y,F_Z)$.
Note that $\low g(x)=1-x-\up\mmpsi(1-x)$ and $\up g(x)=1-x-\low\mmpsi(1-x)$.


\begin{proposition}\label{prop8}
  For every $F_X\in\mathcal{F}_{(\underline{F}_X,\overline{F}_X)}$ and every $F_Y\in\mathcal{F}_{(\underline{F}_Y,\overline{F}_Y)}$ there exist functions $f$ and $g$ associated to the triple $(F_X,F_Y,F_Z)$ such that $\low{f}\leqslant f \leqslant\up{f}$ and  $\low{g}\leqslant g \leqslant\up{g}$.
\end{proposition}

\begin{proof}
  By Lemmas \ref{lem-phi-order}\&\ref{lem-chi-order} we get $\underline{\varphi} \leqslant\overline{\varphi}$ and $\underline{\chi} \leqslant\overline{\chi}$. Using \eqref{eq:rewrite} we get easily the desired result for $f$ and $g$.
\end{proof}

We need another fact, namely that operator $\widehat{\cdot}$ reverses the order. Let us combine all these facts in Equation \eqref{eq:f_g_product} to get for the functions $\underline{f}, \underline{g}, \overline{f}$,and $\overline{g},$
\begin{equation}\label{eq:f_g_bar}
  \begin{split}
     \underline{f}(\underline{F}_XF_Z)=\underline{F}_X\widehat{F}_Z,\quad\quad &\overline{f}(\overline{F}_XF_Z)=\overline{F}_X\widehat{F}_Z, \\
     \underline{g}(\widehat{\overline{F}}_Y\widehat{F}_Z)=\widehat{\overline{F}}_YF_Z, \quad\quad & \overline{g}(\widehat{\underline{F}}_Y\widehat{F}_Z)=\widehat{\underline{F}}_YF_Z,
  \end{split}
\end{equation}
whenever $\underline{F}_XF_Z>0$, $\overline{F}_XF_Z>0$, $\widehat{\underline{F}}_Y \widehat{F}_Z>0$, $\widehat{\overline{F}}_Y \widehat{F}_Z>0$. Here and in the sequel we denote $\widehat{{\low F}}=1-\low F$ and $\widehat{{\up F}}=1-\up F$.
Out of these four equations let us show, say, the southwest one. The others go similarly. When seeking the infimum of the left hand side of the second equation of \eqref{eq:f_g_product}, function $g$ reaches $\underline{g}$, the infimum of the value of $\widehat{F}_Y$ becomes $\widehat{\overline{F}}_Y$ and $\widehat{F}_Z$ remains unchanged; similar considerations apply to the right hand side of the equation.

Before we summarize these observations let us also introduce $H^\sigma(x,y)= P(U\leqslant x, W>y)$ which is playing the role of the combined joint distribution-survival function we need in relation with RMM copulas.
\begin{theorem}[Properties of the imprecise RMM copulas]
 It holds that:
  \begin{enumerate}[(i)]
    \item $\underline{f}\leqslant \overline{f}, \underline{g} \leqslant \overline{g}$.
    \item $\rmmcop_{\underline{f},\underline{g}} \geqslant \rmmcop_{f,g} \geqslant \rmmcop_{\overline{f}, \overline{g}}$.
    \item \[\begin{split}
               \underline{F}_U=\underline{F}_X F_Z,\quad & \overline{F}_U=\overline{F}_XF_Z,  \\
               \widehat{\underline{F}}_W= \widehat{\underline{F}}_Y \widehat{F}_Z,\quad  & \widehat{\overline{F}}_W= \widehat{\overline{F}}_Y \widehat{F}_Z.
            \end{split} \]
    \item $\underline{F}_U\leqslant F_U\leqslant \overline{F}_U, \quad \widehat{\underline{F}}_W\geqslant \widehat{F}_W\geqslant \widehat{\overline{F}}_W$.
    \item \begin{equation*}
            \begin{split}
               \underline{f}(\underline{F}_U)  =\underline{F}_X - \underline{F}_U, \quad
               &\overline{f}(\overline{F}_U)= \overline{F}_X - \overline{F}_U, \\
              \underline{g}(\widehat{\overline{F}}_W)= \widehat{\overline{F}}_Y-\widehat{\overline{F}}_W,\quad &\overline{g}(\widehat{\underline{F}}_W)= \widehat{\underline{F}}_Y-\widehat{\underline{F}}_W,\end{split}
          \end{equation*}
          whenever $\underline{F}_U$, $\overline{F}_U$, $\widehat{\underline{F}}_W$, $\widehat{\overline{F}}_W>0$.
    \item 
            ${\underline f}^*({\underline F_U}) {\underline g}^*({ \widehat{\overline F}_W})={\underline f}^*({\underline F_U}) {\overline g}^*({ \widehat{\underline F}_W})={\overline f}^*({\overline F_U}) {\underline g}^*({\widehat{\overline F}_W}) = {\overline f}^*({\overline F_U}) {\overline g}^*({ \widehat{\underline F}_W})=1$, if 
$\underline{F}_U$, $\overline{F}_U$, $\widehat{\underline{F}}_W$, $\widehat{\overline{F}}_W>0$.
    \item $H^\sigma(x,y)=F_X(x)\widehat{F}_Y(y)\max\{0,F_Z(x)- F_Z(y)\}=\rmmcop_{f,g}(F_U(x),\widehat{F}_W(y))$.
    \item $\underline{H}^\sigma(x,y)=\rmmcop_{\underline{f},\underline{g}} (\underline{F}_U(x),\widehat{\overline{F}}_W(y))$ and $\overline{H}^\sigma(x,y)=\rmmcop_{\overline{f}, \overline{g}}(\overline{F}_U(x), \widehat{\underline{F}}_W(y))$.
    \item $\rmmcop_{\underline{f},\underline{g}} (\underline{F}_U(x), \widehat{\overline{F}}_W(y)) \leqslant \rmmcop_{f,g} (F_U(x), \widehat{F}_W(y)) \leqslant \rmmcop_{\overline{f}, \overline{g}} (\overline{F}_U(x), \widehat{\underline{F}}_W(y))$.
  \end{enumerate}
\end{theorem}

\begin{proof}
  \emph{(i)} follows by Proposition \ref{prop8}.
  \emph{(ii)}: Choose arbitrary $u,v\in[0,1]$. For nonnegative functions with $f(u)\leqslant f'(u)$ and $g(v)\leqslant g'(v)$ we get easily $f(u)g(v) \leqslant f'(u)g'(v)$, so that $uv- f(u)g(v) \geqslant uv-f'(u)g'(v)$, consequently $\rmmcop_{f,g} \geqslant \rmmcop_{f',g'}$, and the desired conclusion follows. \emph{(iii)}: Clear. Now, we have $F_U=F_XF_Z$ by definition, so that $F_X\leqslant F'_X$ implies $F_U=F_XF_Z\leqslant F'_XF_Z=F'_U$ yielding the first two relations of \emph{(iv)}. The other two relations follow from the fact that $F_Y\leqslant F'_Y$ implies $\widehat{F}_W= \widehat{F}_Y \widehat{F}_Z\geqslant \widehat{F'}_Y\widehat{F}_Z= \widehat{F'}_W$. 
  \emph{(v)}: These are Relations \eqref{eq:f_g_bar} rewritten. \emph{(vi)}: Follows directly by using \emph{(iii)} and \emph{(v)}. \emph{(vii)}: Write
  \begin{equation}\label{eq:H}
    \begin{split}
       H^\sigma(x,y) & =P(\max(X,Z)\leqslant x,\min(Y,Z)>y) \\
         & =P(X\leqslant x, Y>y, y<Z\leqslant x) \\
         & =F_X(x)\widehat{F}_Y(y)(F_Z(x)- F_Z(y))
    \end{split}
  \end{equation}
  if $x>y$ and zero otherwise. On the other hand
  \[
    \rmmcop_{f,g}(F_U(x),\widehat{F}_W(y))=F_X(x)F_Z(x)\widehat{F}_Y(y) \widehat{F}_Z(y)-F_X(x)\widehat{F}_Z(x) \widehat{F}_Y(y)F_Z(y)
  \]
  whenever this expression is positive, and zero otherwise; and this amounts to the same as in \eqref{eq:H}.
  \emph{(viii)}: Compute infimum of the leftmost side and on the rightmost side of \eqref{eq:H} to get
  \[
    \underline{H}^\sigma(x,y)=\underline{F}_X(x)\widehat{\overline{F}}_Y(y) \max\{0,F_Z(x)- F_Z(y)\}
  \]
  which implies the first desired relation using \emph{(iii)} and \emph{(v)}. The second one goes similarly.
  \emph{(ix)}: Follows from \emph{(vii)} and \emph{(viii)}.
\end{proof}

\textbf{Remark.} Observe that, somewhat surprisingly, Relation \emph{(ix)} holds in spite of the fact that $\rmmcop_{\underline{f}, \underline{g}} \geqslant \rmmcop_{\overline{f}, \overline{g}}$ as implied by \emph{(ii)}.

\vskip.3cm

\textbf{Example.} 
Suppose the occurrence of endogenous shocks in the model is governed by independent Poisson processes and exogenous shock comes at a fixed future time. Then $X, Y$ and $Z$ are independent random variables with distribution functions:
	\begin{align*}
		F_X(x) & = 1-e^{-\lambda x}, \text{ for } x\geqslant 0 \text{ and 0 for } x<0; \\
		F_Y(y) & = 1-e^{-\mu y}, \text{ for } y\geqslant 0 \text{ and 0 for } y<0; \\
		F_Z(x) & = \begin{cases}
			0 & \text{ if } x< 1; \\
			1 & \text{ if } x\geqslant 1, \\
		\end{cases}
	\end{align*}
	where $\lambda$ and $\mu$ are some positive constants, actually they are the parameters of the underlying Poisson processes. We are normalizing the parameters so that shock $Z$ comes at time 1. Further, the distribution functions of $U = \max\{ X, Z \}$ and $W= \min\{ Y, Z \}$ are equal to
	\begin{align*}
		F_U(x) & = \begin{cases}
			1-e^{-\lambda x} &  \text{ if } x \geqslant 1 ; \\
			0 & \text{ elsewhere};
		\end{cases} \\
		F_W(y) & = \begin{cases}
			0 & \text{ if } y < 0 ; \\
		 	1-e^{-\mu y} & \text{ if } 0\leqslant y< 1 ; \\
		 	1 & \text{ if } 1 \leqslant y .
		\end{cases}		
	\end{align*}
	Reflected maxmin copula $\rmmcop_{f,g}$ modeling the dependence between $U$ and $W$ is generated by the functions
	\begin{align*}
		f(u) & = \max\{1-e^{-\lambda}-u, 0\} , \\
		g(w) & = \max\{e^{-\mu}-w, 0\} 
	\end{align*}
	for $u, w \in (0, 1]$, and $f(0)=g(0)=0$. It is equal to
	$$ \rmmcop_{f,g}(u, w) = \begin{cases}
			0 & \begin{array}{l} \text{if } e^{-\mu}u + (1 - e^{-\lambda})w \leqslant e^{-\mu}(1 - e^{-\lambda}) ; \end{array} \\
		 	e^{-\mu}u + (1 - e^{-\lambda})w - e^{-\mu}(1 - e^{-\lambda}) & \begin{array}{l} \text{if } e^{-\mu}u + (1 - e^{-\lambda})w > e^{-\mu}(1 - e^{-\lambda}), \\
			u \leqslant 1-e^{-\lambda}, w \leqslant e^{-\mu} ; \end{array} \\
		 	uw & \begin{array}{l} \text{if } u > 1-e^{-\lambda} \text{ or } w > e^{-\mu}. \end{array}
		\end{cases}		$$
		Suppose now that we cannot assume precisely given parameters, but instead we consider the $p$-boxes $(\low F_X, \up F_X)$ and $(\low F_Y, \up F_Y)$, where $\low F_X(x)$ is an exponential distribution with parameter $\lambda_1$ and $\up F_X(x)$ with some parameter $\lambda_2 > \lambda_1$. It is immediate that $\low F_X \leqslant \up F_X$ holds. Similarly, let $\low F_Y$ and $\up F_Y$ be exponential with parameters $\mu_1 < \mu_2$ respectively. It is easy to check that 	
		\begin{align*}
		\underline{f}(u) & = \max\{1-e^{-\lambda_1}-u, 0\} , \\
		\overline{f}(u) & = \max\{1-e^{-\lambda_2}-u, 0\} , \\
		\underline{g}(w) & = \max\{e^{-\mu_2}-w, 0\} , \\
		\overline{g}(w) & = \max\{e^{-\mu_1}-w, 0\} 
	\end{align*}
for $u, w \in (0, 1]$ and $0$ otherwise are the generating functions of copulas $\rmmcop_{\underline{f},\underline{g}}$ and $\rmmcop_{\overline{f},\overline{g}}$. Notice that the order of functions $g$ is reversed with respect to the order of the parameters $\mu$. 
In Figure \ref{fig-1} we give the 3D graphs of copulas $\rmmcop_{\underline{f},\underline{g}}$ and $\rmmcop_{\overline{f},\overline{g}}$ for the parameters $\lambda_1= \mu_1 = 1$ and $\lambda_2= \mu_2 = 2$, where the relation $\rmmcop_{\underline{f},\underline{g}} \geqslant \rmmcop_{\overline{f},\overline{g}}$ can be seen. 

\begin{figure}[h]
            \includegraphics[width=6cm]{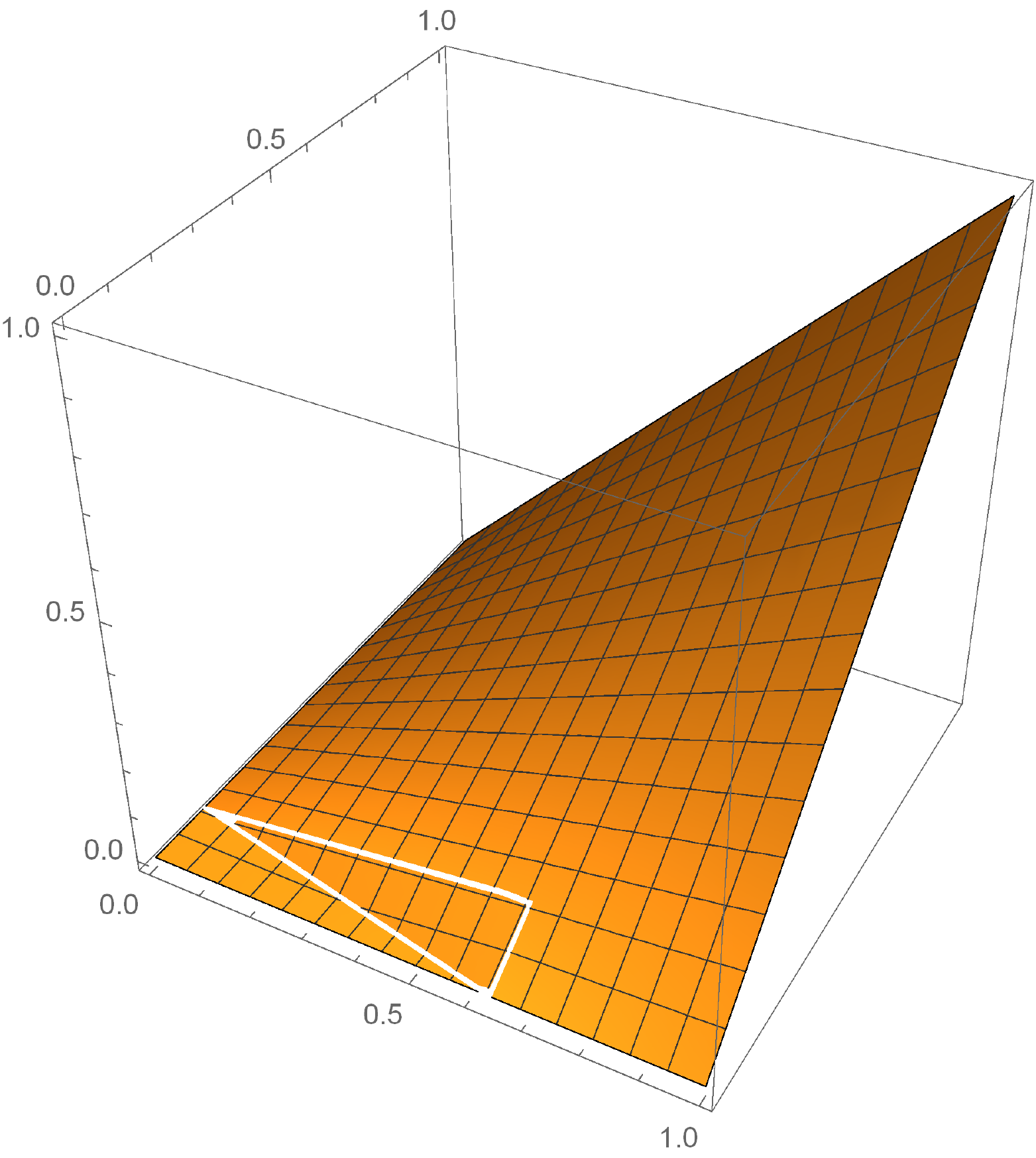} \hfil \includegraphics[width=6cm]{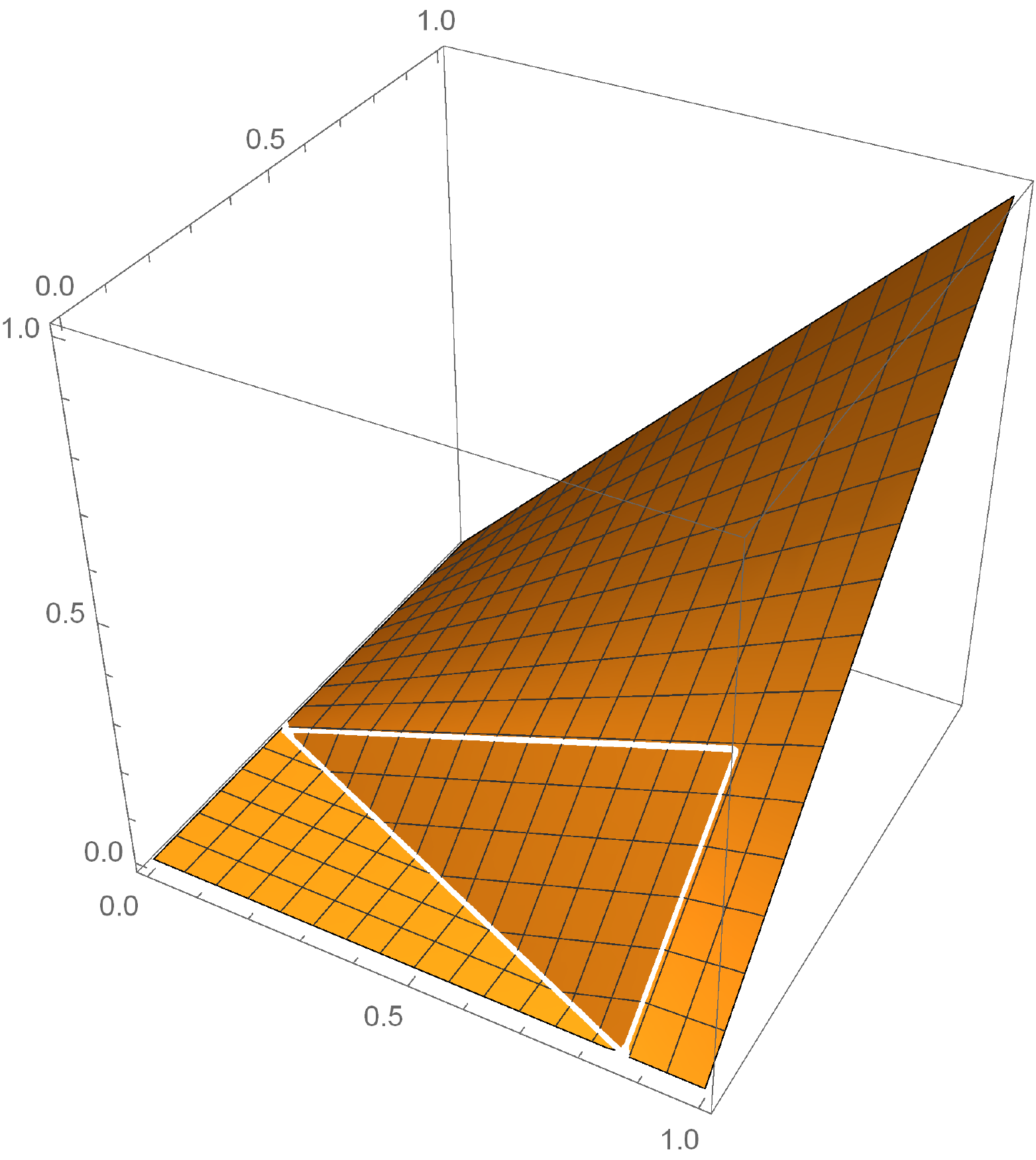} 
            \caption{3D graphs of copulas $\rmmcop_{\underline{f},\underline{g}}$ and $\rmmcop_{\overline{f},\overline{g}}$ for the parameters $\lambda_1= 1, \mu_2 = 2$ (left) and $\lambda_2= 2, \mu_1 = 1$ (right).} \label{fig-1}
\end{figure}

\vskip.5cm

\section{Multivariate imprecise Marshall's copulas}\label{subsec:imp:marsh}

In this section we extend the bivariate imprecise Marshall's copulas described in Subsection \ref{subsec:imp:marmm} to the multivariate case.
Start by revisiting the precise case (see for example \cite{DuGiMa} and references therein). Let $X_1,\ldots,X_n,Z$ be independent variables with respective distribution functions $F_1,\ldots,F_n,F_Z$. Define
\begin{equation}\label{eq:maxmax}
  U_i=\max\{X_i,Z\},\ \ \mbox{for}\ i=1,\ldots,n.
\end{equation}
Then for the respective distribution functions $G_1, \ldots, G_n$ of $U_1, \ldots, U_n$ we have clearly
\[
    G_i=F_iF_Z\quad\mbox{for}\quad i=1,2,\ldots,n.
\]
Denote by $H(x_1,\ldots,x_n)$ the joint distribution function of the random vector $(U_1,U_2,\ldots,U_n)$. Following the notation of the bivariate case presented in Subsection \ref{subsec:marsall} we introduce the generating functions so that they satisfy Conditions (F1)--(F3) and the defining relations
\[
    \varphi_i(G_i)=F_i\quad\mbox{if}\quad G_i>0\quad \mbox{for}\quad i=1,2,\ldots,n.
\]
Note that these relations do not determine the generating functions uniquely. Recall that the joint distribution function equals
\[
    H(x_1,\ldots,x_n)= F_1(x_1)\cdots F_n(x_n) F_Z(\min\{x_1,\ldots,x_n\}),
\]
and that
\begin{equation}\label{minn}
  C(u_1,\ldots,u_n)=\varphi_1(u_1)\cdots \varphi_n(u_n) \min\left\{\dfrac{u_1}{\varphi_1(u_1)},\ldots, \dfrac{u_n}{\varphi_n(u_n)}\right\},
\end{equation}
if all $\varphi_i(u_i)>0$ and zero otherwise. We will later introduce the notation $\marcop_{\pmb{\varphi}}$ for this copula. After a straightforward computation, one concludes that $C$ is a copula such that
\begin{equation}\label{eq:sklarn}
  H(x_1,\ldots,x_n)=C(G_1(x_1),\ldots,G_n(x_n)).
\end{equation}
Observe that we have thus extended an alternative version of the usually preferred Marshall's formula from the bivariate case (we presented both versions in Subsection \ref{subsec:marsall}) to the $n$-variate case.

In the imprecise setting we need to be more careful. From now on in this subsection all our distribution functions are assumed to come from a finitely-additive probability space meaning that they are monotone only. So, the random variables representing endogenous shocks are assumed to be given by distributions $F_i\in \mathcal{F}_{(\underline{F}_i, \overline{F}_i)}$ for $i=1,2,\ldots, n$, and by the Marshall's assumption \eqref{eq:maxmax} the random variables corresponding to the lives of the components satisfy $G_i=F_iF_Z$  for $i=1,2,\ldots, n$. For any choice of marginal distributions we define the joint distribution function $H(x_1,\ldots, x_n) = F_1(x_1)\cdots F_n(x_n) F_Z(\min\{x_1,\ldots,x_n\})$ by analogy with the above. This implies that the minimal and maximal joint distribution functions
\[\begin{split}
     \low H(x_1,\ldots, x_n) & =\min\{ H(x_1,\ldots, x_n)~|~F_i\in \mathcal{F}_{(\underline{F}_i, \overline{F}_i)}~\mbox{for}~i=1,2,\ldots, n\}\ \mbox{and} \\
     \up H(x_1,\ldots, x_n)  & =\max\{ H(x_1,\ldots, x_n)~|~F_i\in \mathcal{F}_{(\underline{F}_i, \overline{F}_i)}~\mbox{for}~i=1,2,\ldots, n\}
  \end{split}
\]
are clearly equal to
\begin{equation}\label{eq:H6.1}
  \begin{split}
     \low H(x_1,\ldots, x_n) & =\low F_1(x_1)\cdots \low F_n(x_n) F_Z(\min\{x_1,\ldots,x_n\})\ \mbox{and} \\
     \up H(x_1,\ldots, x_n)  & =\up F_1(x_1)\cdots\up F_n(x_n) F_Z(\min\{x_1,\ldots,x_n\}).
  \end{split}
\end{equation}
\vskip 2mm
We define the corresponding generating functions $\varphi_i$ using Equation \eqref{eq-phi-ext} in which we are consecutively replacing $F_X$ by $F_i$ for $i=1,2,\ldots,n$. By Proposition \ref{prop-marshall-properties}\emph{(ii)} we have
\[
    \varphi_i(G_i)=F_i,\quad\mbox{if}\quad G_i>0, \quad\mbox{for}\quad i=1,2,\ldots,n,
\]
in agreement with the defining relations above. By Lemma \ref{lem-phi-order} we deduce that
\[
    F_i'\leqslant F_{i}\quad\mbox{implies}\quad\varphi'_i\leqslant\varphi_i.
\]

Introduce the vectors $\underline{\pmb{\varphi}}= (\underline{\varphi}_1, \underline{\varphi}_2, \ldots, \underline{\varphi}_n)$, $\overline{\pmb{\varphi}}= (\overline{\varphi}_1,\overline{\varphi}_2, \ldots, \overline{\varphi}_n)$ and $\pmb{\varphi}= (\varphi_1,\varphi_2, \ldots, \varphi_n)$
. For the copula of Equation \eqref{minn} introduce notation
\[
    \marcop_{\pmb{\varphi}}(\pmb{u}) =\prod_{i=1}^{n}\varphi_i(u_i) \min_{i=1,\ldots,n}\left\{\frac{u_i}{\varphi_i(u_i)}\right\}
\]
if $\varphi_i(u_i)>0$ for all $i=1,2,\ldots,n$, and zero otherwise. We rewrite also the Equation \eqref{eq:sklarn} into
\begin{equation}\label{sklarn'}
  H(x_1,\ldots,x_n)= \marcop_{\pmb{\varphi}}(G_1(x_1),\ldots, G_n(x_n)).
\end{equation}
Given two vectors of functions $\underline{\pmb{\varphi}}\leqslant \overline{\pmb{\varphi}}$ (here and in what follows the relation ``less than or equal to'' is meant componentwise) such that each of their components satisfies Conditions (P1)--(P3), we let
\[
    \marsetcop
    =\{\marcop_{\pmb{\varphi}}\colon \underline{\pmb{\varphi}}\leqslant\pmb{\varphi}\leqslant \overline{\pmb{\varphi}}\},
\]
where each component of $\pmb{\varphi}$ also satisfies Conditions (P1)--(P3), and call this set of copulas an \emph{$n$-variate imprecise Marshall's copula}.

Using the ideas of Subsection \ref{subsec:imp:marmm} let $\low \varphi_i$ respectively $\up{\varphi}_i$ be the minimal respectively the maximal function satisfying Conditions (P1)--(P3) and
\[
    \underline{\varphi}_i(\underline{G}_i)=\underline{F}_i\ \ \mbox{if}\ \ \underline{G}_i>0\quad \mbox{respectively}\quad \overline{\varphi}_i(\overline{G}_i) = \overline{F}_i\ \ \mbox{if}\ \ \overline{G}_i>0
\]
for $i=1,2,\ldots,n$. Define also functions ${\underline \varphi_i}^*$ and ${\overline \varphi_i}^*$ as in Condition (P3).

Let us summarize.

\begin{theorem}[Properties of multivariate imprecise Marshall's copulas]\label{thm:marshall multi}
	In the situation described above we have:
  \begin{enumerate}[(i)]
    \item $\underline{\pmb{\varphi}}
        \leqslant \overline{\pmb{\varphi}}$.
    \item $\marcop_{\underline{\pmb{\varphi}}}\leqslant \marcop_{\pmb{\varphi}}\leqslant \marcop_{\overline{\pmb{\varphi}}}$.
    \item \[G_i=F_iF_Z,\quad \underline{G}_i=\underline{F}_iF_Z,\quad \overline{G}_i=\overline{F}_iF_Z, \] for $i=1,2,\ldots,n$.
    \item \[\begin{split}
               \varphi_i(G_i)=F_i\quad &\mbox{if}\quad G_i>0,  \\
               \underline{\varphi}_i(\underline{G}_i)=\underline{F}_i\quad & \mbox{if}\quad \underline{G}_i>0, \\
               \overline{\varphi}_i(\overline{G}_i)=\overline{F}_i\quad & \mbox{if}\quad \overline{G}_i>0,
            \end{split}\] for $i=1,2,\ldots,n$.
    \item $\marcop_{\underline{\pmb{\varphi}}}(\underline{G}_1(x_1),\ldots, \underline{G}_n(x_n)) \leqslant \marcop_{\pmb{\varphi}} (G_1(x_1),\ldots, G_n(x_n)) \leqslant \marcop_{\overline{\pmb{\varphi}}} (\overline{G}_1(x_1),\ldots, \overline{G}_n(x_n))$.
    \item \[\begin{split}
               H(x_1,\ldots,x_n) & = \marcop_{\pmb{\varphi}}(G_1(x_1),\ldots, G_n(x_n)), \\
               \underline{H}(x_1,\ldots,x_n)  & = \marcop_{\underline{\pmb{\varphi}}}(\underline{G}_1(x_1),\ldots, \underline{G}_n(x_n)), \\
               \overline{H}(x_1,\ldots,x_n)  & = \marcop_{\overline{\pmb{\varphi}}}(\overline{G}_1(x_1),\ldots, \overline{G}_n(x_n)).
            \end{split}
        \]
        \item For all $i,j=1, \ldots, n$ we have ${\underline \varphi_i}^*({\underline G_i}) ={\overline \varphi_j}^*({\overline G_j})$ if ${\underline G_i}, {\overline G_j} >0$.
  \end{enumerate}
\end{theorem}

\begin{proof}
  \emph{(i)} follows by the observations above, \emph{(ii)} follows from \emph{(i)} and Equation \eqref{minn}. \emph{(iii)} and \emph{(iv)} follow by definition. \emph{(v)}
  First,
  \[
  \marcop_{\underline{\pmb{\varphi}}} (\underline{G}_1(x_1),\ldots, \underline{G}_n(x_n)) \leqslant \marcop_{\underline{\pmb{\varphi}}}({G}_1(x_1),\ldots, {G}_n(x_n))
  \]
  since $\underline{G}_i=\low F_iF_Z\leqslant F_iF_Z=G_i$ and $\marcop_{\underline{\pmb{\varphi}}}$ is monotone. Next, by \emph{(ii)} we get
  $$\marcop_{\underline{\pmb{\varphi}}}({G}_1(x_1),\ldots, {G}_n(x_n))\leqslant \marcop_{{\pmb{\varphi}}} ({G}_1(x_1),\ldots, {G}_n(x_n)).$$
  The other side of the inequality goes similarly. Point \emph{(vi)} follows from Equations \eqref{eq:sklarn} and \eqref{eq:H6.1} and point \emph{(vii)} follows from {\emph{(iv)}}.
\end{proof}

\vskip.5cm

\section{Multivariate imprecise maxmin copulas}\label{subsec:imp:maxmin}

In this section we extend the bivariate imprecise maxmin copulas described in Subsection \ref{subsec:imp:marmm} to the multivariate case.
Start again by revisiting the precise case (see \cite{DuOmOrRu}).  We let $X_1,\ldots,X_p,X_{p+1}, \ldots,X_n,Z$
be independent variables with respective distribution functions $F_1,\ldots,F_p,$ $F_{p+1},\ldots,F_n,F_Z$ and define
\begin{equation}\label{eq:rmm:0}
  \begin{split}
     U_i & =\max\{X_i,Z\},\ \mbox{for}\ i=1,\ldots,p,\ \ \mbox{and} \\
     U_j & =\min\{X_j,Z\},\ \mbox{for}\ j=p+1,\ldots,n.
  \end{split}
\end{equation}
So, for the respective distribution functions $G_1, \ldots, G_n$ of $U_1, \ldots, U_n$ we have
\[
    \begin{split}
       G_i & =F_iF_Z\quad\mbox{for}\quad i=1,\ldots,p,\ \ \mbox{and} \\
       \widehat{G}_j & =\widehat{F}_j\widehat{F}_Z\quad\mbox{for}\quad j=p+1,\ldots,n.
    \end{split}
\]
Following the elaboration of the bivariate case presented in Subsection \ref{subsec:maxmin} (and in particular Proposition \ref{prop-maxmin-properties}) we deduce that this time the defining relations of the generating functions, necessarily satisfying Conditions (F1)--(F3), are given by
\[
    \begin{alignedat}{2}
       \varphi_i(G_i)&=F_i \quad \mbox{if}\quad G_i>0\quad \mbox{for}\quad &i=1,\ldots,p,\ \mbox{and} \\
       \chi_j({G}_j)&=F_j \quad\mbox{if}\quad G_j<1\quad \mbox{for} &j=p+1,\ldots,n.
    \end{alignedat}
\]
For a vector $(x_1,\ldots,x_n)\in\mathds{R}^n$ define
\[
    H(x_1,\ldots,x_n)=P(U_1\leqslant x_1,\ldots,U_n\leqslant x_n)
\]
and exploit Formula (4.3) of \cite{DuOmOrRu} to get for the independent case
\begin{equation}\label{eq:orazem}
  \begin{split}
     H(x_1,\ldots,x_n) &=\sum_{K\subseteq S}\prod_{i\in T\cup (S\setminus K)}F_i(x_i)\times \\
       & \max\left\{0,F_Z\left(\min_{i\in T\cup K} x_i \right) -F_Z\left(\max_{j\in S\setminus K} x_j \right)\right\},
  \end{split}
\end{equation}
where $T=\{1,2,\ldots,p\}$, $S=\{p+1,p+2,\ldots,n\}$, and $K$ runs through all the subsets of $S$. If we want to do the imprecise case, we need to have a deeper understanding of this formula. We first write $H(x_1,\ldots,x_n)=P(A)$ where $A= \bigcap_{i=1}^{n}A_i$ and $A_i=(U_i\leqslant x_i)$, for $i\in T\cup S$. Moreover, for $j\in S$ define $B_j= A_j^\textup{c}=\Omega\setminus A_j$ so that $B_j=(U_j>x_j)$. It follows that
\[
    A=\left(\bigcap_{i\in T}A_i\right)\cap\left(\bigcup_{j\in S}B_j\right)^\textup{c},
\]
and by the inclusion-exclusion principle
\[
  \begin{split}
     P(A) & =P\left(\bigcap_{i\in T}A_i\right)-P\left(\left(\bigcap_{i\in T}A_i\right) \cap \left( \bigcup_{j\in S} B_j\right)\right) \\
       & =P\left(\bigcap_{i\in T}A_i\right)-\sum_{\emptyset\neq K\subseteq S} (-1)^{|K|+1} P\left(\left(\bigcap_{i\in T}A_i\right) \cap \left( \bigcap_{j\in K}B_j\right)\right).
  \end{split}
\]
Observe that $A_i=(\max\{X_i,Z\}\leqslant x_i)=(X_i\leqslant x_i)\cap(Z\leqslant x_i)$ for $i\in T$ and that $B_j=(\min\{X_j,Z\}> x_j)=(X_j> x_j)\cap(Z> x_j)$ for $j\in S$, so that, using the independence assumption
\[
  \begin{split}
     P(A)
       & =\prod_{i\in T}F_i(x_i)P\left(Z\leqslant\min_{i\in T} x_i\right)\left(1-\sum_{\emptyset\neq K\subseteq S} (-1)^{|K|+1} P\left( \bigcap_{j\in K}B_j\right)\right) \\
			 & =\prod_{i\in T}F_i(x_i)P\left(Z\leqslant\min_{i\in T} x_i\right)\left(\sum_{K\subseteq S} (-1)^{|K|} P\left( \bigcap_{j\in K}B_j\right)\right) \\
       & =\prod_{i\in T}F_i(x_i)\left(\sum_{K\subseteq S} (-1)^{|K|} \prod_{j\in K} \widehat{F}_j(x_j)
       P\left(\max_{j\in K} x_j <Z\leqslant\min_{i\in T} x_i \right) \right).
  \end{split}
\]
By reordering variables $X_{p+1},\ldots,X_n$, if necessary, let us order the members of the set $\{x_{p+1},\ldots,x_n\}$ so that $x_{p+1} \leqslant x_{p+2}\leqslant \ldots\leqslant x_{n}$. We also introduce $y= \min_{i\in T} x_i$ and choose $\alpha$ to be the largest index with $p <\alpha\leqslant n$ and $x_{\alpha}\leqslant y$. In case that $y<x_{p+1}$ we choose $\alpha=p$. For a $K\subseteq S, K\neq\emptyset,$ denote by $r$ the greatest index contained in $K$ and let $K'\subseteq\{p+1,\ldots,r-1\}$ be such that $K=K'\cup\{r\}$. If $\alpha<r$ then the term of the above sum corresponding to $K$ is zero, so that we can keep  only those terms for which $\alpha\geqslant r$. Therefore,
\[
     P(A)
        =\prod_{i\in T}F_i(x_i) \left( P(Z\leqslant y)+
        \sum_{r=p+1}^{\alpha}E_r \widehat{F}_r(x_r)
       P\left( x_r <Z\leqslant y \right) \right),
\]
where $E_r=1$ if $r=p+1$, and $E_r$ stands for the expression
\[
    E_r=   \sum_{K'\subseteq \{p+1,\ldots,r-1\}} (-1)^{|K'|+1} \prod_{j\in K'} \widehat{F}_j(x_j),
\]
otherwise. Clearly, this is an alternating sum of the elementary symmetric polynomials which is known to be equal to
\[
    E_r=-\prod_{j=p+1}^{r-1}(1- \widehat{F}_j(x_j)) =-\prod_{j=p+1}^{r-1}F_j(x_j),
\]
so that
\[
  \begin{split}
     P(A)
        &=\prod_{i\in T}F_i(x_i) \left( P(Z\leqslant y)-
        \sum_{r=p+1}^{\alpha}\prod_{j=p+1}^{r-1}F_j(x_j) (1-{F}_r(x_r))
       P\left( x_r <Z\leqslant y \right) \right)\\
       &=\prod_{i\in T}F_i(x_i) \left( P(Z\leqslant x_{p+1})+
        \sum_{r=p+1}^{\alpha-1}\prod_{j=p+1}^{r}F_j(x_j)
       P\left( x_r <Z\leqslant x_{r+1} \right) \right.\\
       &\phantom{.}\quad\quad\quad\quad\left.+\prod_{j=p}^{\alpha}F_j(x_j)P\left( x_\alpha <Z\leqslant y \right)\right);
  \end{split}
\]
in order to get the sum in the second expression, we subtract the subtrahend of a certain term from the minuend of the previous term of the sum in the first expression, going through all terms. We now rewrite this sum and compute further
\[
  \begin{split}
     P(A)
       &=\prod_{i\in T}F_i(x_i) \left(F_Z(x_p)+
        \sum_{r=p}^{\alpha}\prod_{j=p}^{r}F_j(x_j)
        \right.\times\\
       &\phantom{.}\quad\quad\quad\quad\times\left. \left(F_Z\left( \min_{j\in T\cup\{r+1,\ldots,n\}}x_j \right)-F_Z\left( \max_{j\in \{1,\ldots,r\}}x_j \right)\right)\right)\\
       &=\prod_{i\in T}F_i(x_i) \left(
        \sum_{K\subseteq S}\prod_{j\in K}F_j(x_j) \max\left\{0,F_Z\left( \min_{j\in T\cup S\setminus K}x_j \right)-F_Z\left( \max_{j\in K}x_j \right)\right\}
        \right).
  \end{split}
\]
Now, the last sum above is written independently of the ordering of the components $x_j$ and this brings us immediately to the desired formula \eqref{eq:orazem}.

We now introduce the auxiliary generating functions. Here we are facing the dilemma that two ways of defining $\varphi^*$ have been used in the literature. To avoid confusion we introduce a new notation
\[\begin{split}
     \varphi^\dag_i(u_i)=\frac{u_i}{\varphi_i(u_i)}\quad \mbox{for}\quad &i=1,\ldots,p, u_i>0\ \mbox{and} \\
     \varphi^\dag_j(u_j)=\frac{u_j-\chi_j(u_j)}{1-\chi_j(u_j)}  \quad \mbox{for}\quad &j=p+1,\ldots,n, u_j<1.
  \end{split}
\]
With this notation of $\varphi^\dag$ we follow the definition of $\varphi^*$ and $\chi_*$ of \cite{DuOmOrRu} and \cite{KoBuKoMoOm2}, while the according definitions of these auxiliary functions were somewhat different in \cite{OmRu}. Next we define
\[
    \pmb{\varphi}=(\varphi_1,\ldots,\varphi_p,\chi_{p+1},\ldots,\chi_{n})\quad \mbox{and}\quad \pmb{\varphi}^\dag=(\varphi_1^\dag,\ldots,\varphi_{n}^\dag).
\]
Formula \eqref{eq:orazem} now yields (cf.\ also \cite[(4.4)]{DuOmOrRu})
\begin{equation}\label{eq:orazem'}
  \begin{split}
     \mmcop_{\pmb{\varphi}}(\mathbf{u}) &=\left(\prod_{i\in T}\varphi_i(u_i) \right)\times \\ \sum_{K\subseteq S}\prod_{j\in S\setminus K}\chi_j(u_j)       &
     \max\left\{0,\min_{i\in T\cup K} \varphi_i^\dag(u_i) -\max_{j\in S\setminus K} \varphi_j^\dag(u_j) \right\},
  \end{split}
\end{equation}
and $H(x_1,\ldots,x_n)=\mmcop_{\pmb{\varphi}}(G_1(x_1),\ldots,G_n(x_n))$.

In the imprecise setting we work in a finitely-additive probability space. Endogenous shocks are given by random variables whose distributions $F_i$ belong to $\mathcal{F}_{(\underline{F}_i, \overline{F}_i)}$. We define the generating functions using the ideas of Equations \eqref{eq-phi-ext} and \eqref{eq-chi-ext} so that, in particular, they do suffice the above defining relations. In addition, by Lemmas \ref{lem-phi-order} and \ref{lem-chi-order} we deduce that
\[
    \begin{alignedat}{4}
       F_i'&\leqslant F_{i}\ \hbox{implies}\ \varphi'_i &\leqslant \varphi_i\ \quad \mbox{for}\ \ &i=1,\ldots,p,\quad\quad \mbox{and} \\
       F_j'&\leqslant F_{j}\ \hbox{implies}\ \chi'_j &\leqslant \chi_j\ \quad \mbox{for}\ &j=p+1, \ldots,n.\phantom{\mbox{and}}
    \end{alignedat}
\]
As in Section \ref{subsec:imp:marsh} we introduce the minimal and the maximal joint distribution functions
\[\begin{split}
     {\low H} & =\min\{ H(x_1,\ldots, x_n)~|~F_i\in \mathcal{F}_{(\underline{F}_i, \overline{F}_i)}~\mbox{for}~i=1,2,\ldots, n\}\ \mbox{and} \\
     {\up H} & =\max\{ H(x_1,\ldots, x_n)~|~F_i\in \mathcal{F}_{(\underline{F}_i, \overline{F}_i)}~\mbox{for}~i=1,2,\ldots, n\}.
  \end{split}
\]

Given two vectors of functions $\underline{\pmb{\varphi}}\leqslant \overline{\pmb{\varphi}}$ such that each of their components satisfies Conditions (F1)--(F3), we let
\[
    \mmsetcop
    =\{\mmcop_{\pmb{\varphi}}\colon \underline{\pmb{\varphi}}\leqslant \pmb{\varphi} \leqslant \overline{\pmb{\varphi}}\},
\]
where each component of $\pmb{\varphi}$ also satisfies Conditions (F1)--(F3), and call this set of copulas an \emph{$n$-variate imprecise maxmin (MM for short) copula}. In Condition (F3) we apply the respective definitions of $\varphi_i^\dag$ and $\varphi_j^\dag$ given above for $i=1,\ldots,p,$ and for $j=p+1, \ldots,n$ instead of starred functions of Subsection \ref{subsec:maxmin}.

We continue to use the ideas of Subsection \ref{subsec:maxmin} by letting $\low \varphi_i$ and $\low \chi_j$ respectively $\up{\varphi}_i$ and $\up{\chi}_j$ be the minimal respectively the maximal function satisfying Conditions (F1)--(F3) and
\[
    \begin{alignedat}{3}
       \low \varphi_i(\low G_i) &=\low F_i \  \mbox{if}\  \low G_i>0 \ \mbox{resp.}\ \up \varphi_i(\up G_i) &=\up F_i \  \mbox{if}\  \overline{G}_i>0\ \mbox{for}\ i&=1,\ldots,p,\quad \mbox{and} \\
       \low \chi_j({\low G}_j) &={\low F}_j \  \mbox{if}\ \low G_j<1\ \mbox{resp.}\ \up \chi_j({\up G}_j) &={\up F}_j \  \mbox{if}\  \up{G}_j<1\ \mbox{for}\ j&=p+1,\ldots,n.
    \end{alignedat}
\]
Let us summarize.
\vskip2mm

\begin{theorem}[Properties of multivariate imprecise MM copulas]\label{thm:maxmin multi}\label{thm:11}
	In the situation described above we have:
  \begin{enumerate}[(i)]
    \item $\underline{\pmb{\varphi}} \leqslant \overline{\pmb{\varphi}}$.
    \item \[
    \begin{alignedat}{3}
       G_i & =F_iF_Z,\quad \underline{G}_i&=\underline{F}_iF_Z,\quad \overline{G}_i&=\overline{F}_iF_Z,\quad\mbox{for}\quad i=1,\ldots,p,\ \ \mbox{and} \\
       \widehat{G}_j & =\widehat{F}_j\widehat{F}_Z,\quad \widehat{\low G}_j & =\widehat{\low F}_j\widehat{F}_Z,\quad \widehat{\up G}_j & =\widehat{\up F}_j \widehat{F}_Z,\quad\mbox{for}\quad j=p+1,\ldots,n.
    \end{alignedat}
    \]
    \item \[
    \begin{split}
       \low G_i\leqslant G_i\leqslant \up G_i & \quad\mbox{for}\quad i=1,\ldots,p,\ \ \mbox{and} \\
       {\low G}_j\leqslant {G}_j\leqslant {\up G}_j & \quad\mbox{for}\quad j=p+1,\ldots,n.
    \end{split}
    \]
    \item \[
    \begin{alignedat}{3}
       \low \varphi_i(\low G_i) &=\low F_i \  \mbox{if}\  \low G_i>0 \ \mbox{resp.}\ \up \varphi_i(\up G_i) &=\up F_i \  \mbox{if}\  \overline{G}_i>0\ \mbox{for}\ i&=1,\ldots,p,\quad \mbox{and} \\
       \low \chi_j({\low G}_j) &={\low F}_j \  \mbox{if}\ \low G_j<1\ \mbox{resp.}\ \up \chi_j({\up G}_j) &={\up F}_j \  \mbox{if}\  \up{G}_j<1\ \mbox{for}\ j&=p+1,\ldots,n.
    \end{alignedat}
    \]
    \item $\varphi^\dag_i(G_i)=\up \varphi^\dag_j(\up G_j)=\low \varphi^\dag_k( \low G_k)=F_Z$ for all $i,j,k=1,\ldots,n$.
    \item \[
        H(x_1, \ldots,x_n)=\mmcop_{\pmb{\varphi}}(G_1(x_1),\ldots, {G}_n(x_n)).
        \]
    \item \[\begin{split}
                 & \mmcop_{\low{{\pmb{\varphi}}}}(\low{G}_1(x_1),\ldots, {\low G}_n(x_n)) \\
               \leqslant\  & \mmcop_{{{\pmb{\varphi}}}} (G_1(x_1),\ldots,{{G} }_n(x_n)) \\
               \leqslant\  & \mmcop_{\up{{\pmb{\varphi}}}} (\up G_1(x_1),\ldots,{\up{G} }_n(x_n)).
            \end{split}
    \]
    \item \[ \begin{split}
                \low H(x_1, \ldots,x_n) & =\mmcop_{\low{{\pmb{\varphi}}}}(\low{G}_1(x_1),\ldots, {\low G}_n(x_n)),                \\
                \up H(x_1, \ldots,x_n) & =\mmcop_{\up{{\pmb{\varphi}}}} (\up G_1(x_1),\ldots,{\up{G} }_n(x_n)).
             \end{split}
        \]
  \end{enumerate}
\end{theorem}

\begin{proof}
  Points from \emph{(i)} to \emph{(iv)} follow by the observations above. \emph{(v)}: Let us compute only one of the three cases that go in a similar way:
  \[
    \varphi_i^\dag(G_i)=\frac{G_i}{\varphi_i (G_i)}=\frac{ G_i}{ F_i}= {F_Z},\ \mbox{for}\ i=1,\ldots,p,
  \]
  and
  \[
    \varphi_j^\dag(G_j)=\frac{ G_j-\chi_j(G_j)}{1-\chi_j(G_j)}= \frac{ G_j-F_j}{1-F_j}={F}_Z,\ \mbox{for}\ j=p+1,\ldots,n.
  \]
  Point \emph{(vi)} amounts to the same as Equation \eqref{eq:orazem'}. To get \emph{(vii)} observe that
  \[
    \mmcop_{\low{{\pmb{\varphi}}}}(\low{G}_1(x_1),\ldots, {\low G}_n(x_n))=\]\[\sum_{K\subseteq S}\prod_{i\in T\cup (S\setminus K)}\low F_i(x_i)
     \max\left\{0,\min_{i\in T\cup K} \low\varphi_i^\dag(\low G_i(x_i)) -\max_{j\in S\setminus K} \low\varphi_j^\dag(\low G_j(x_j))\right\}= \]
  \[
    \sum_{K\subseteq S}\prod_{i\in T\cup (S\setminus K)}\low F_i(x_i)
     \max\left\{0,\min_{i\in T\cup K} F_Z(x_i) -\max_{j\in S\setminus K} F_Z(x_j)\right\}\leqslant\]
  \[
    \sum_{K\subseteq S}\prod_{i\in T\cup (S\setminus K)}F_i(x_i)
     \max\left\{0,\min_{i\in T\cup K} F_Z(x_i) -\max_{j\in S\setminus K} F_Z(x_j)\right\}=\mmcop_{{{\pmb{\varphi}}}} (G_1(x_1),\ldots,{{G} }_n(x_n)).\]
and the first desired inequality follows. Considerations of the same kind yield the second one. Finally, Point \emph{(viii)} follows from Points \emph{(vi)} and \emph{(vii)}.
\end{proof}

\vskip.5cm

\section{Multivariate imprecise RMM copulas}\label{subsec:imp:rmm}

Finally, we extend the imprecise reflected maxmin copulas from Section \ref{sec:reflected-bi} to the multivariate case as well.
For the third time we start by revisiting the precise case.
As in Section \ref{subsec:imp:maxmin} we let \[X_1,\ldots,X_p,X_{p+1}, \ldots,X_n,Z\] be independent variables with respective distribution functions \[F_1,\ldots,F_p,F_{p+1},\ldots,F_n,F_Z\] and define
\begin{equation}\label{eq:rmm:0}
  \begin{split}
     U_i & =\max\{X_i,Z\},\ \mbox{for}\ i=1,\ldots,p,\ \ \mbox{and} \\
     U_j & =\min\{X_j,Z\},\ \mbox{for}\ j=p+1,\ldots,n.
  \end{split}
\end{equation}
So, for the respective distribution functions $G_1, \ldots, G_n$ of $U_1, \ldots, U_n$ we have again
\[
    \begin{split}
       G_i & =F_iF_Z\quad\mbox{for}\quad i=1,\ldots,p,\ \ \mbox{and} \\
       \widehat{G}_j & =\widehat{F}_j\widehat{F}_Z\quad\mbox{for}\quad j=p+1,\ldots,n.
    \end{split}
\]
Following the notation of the bivariate case presented in Subsection \ref{subsec:reflected} we determine that the generating functions should suffice
\[
    \begin{alignedat}{3}
       \varphi_i(G_i) & =F_i,\quad f_i(G_i) &=F_i\widehat{F}_Z \quad \mbox{if}\quad G_i>0\quad \mbox{for}\quad i&=1,\ldots,p,\quad \mbox{and} \\
       \chi_j(G_j) & =F_j,\quad f_j(\widehat{G}_j) &=\widehat{F}_jF_Z \quad \mbox{if}\quad G_j<1\quad \mbox{for}\quad j&=p+1,\ldots,n.
    \end{alignedat}
\]
Following the notation of the bivariate case write $H^\sigma(x_1, \ldots,x_n)$ for the joint distribution function of the random vector $(U_1,U_2,\ldots,U_n)$ in which the last $n-p$ entries are reflected, i.e.,
\[
    H^\sigma(x_1, \ldots,x_n)=P(U_1\leqslant x_1,\ldots,U_p\leqslant x_p, U_{p+1}>x_{p+1},\ldots,U_n>x_n).
\]
Recall \cite[Theorem 14]{KoBuKoMoOm2} to get
\begin{equation}\label{eq:rmm:1}
    H^\sigma(x_1, \ldots,x_n)=\rmmcop_\mathbf{f}(G_1(x_1),\ldots, G_p(x_p),\widehat{G}_{p+1}(x_{p+1}),\ldots,\widehat{G}_n(x_n)).
\end{equation}
There the authors introduced notation
\begin{equation}\label{eq:rmm:2}
  \rmmcop_\mathbf{f}(\mathbf{u})=
\max \left\{0, \min_{\substack{i\in \{1,\ldots,p\} \\ j \in \{p+1,\ldots,n\}}} \left( \left(u_iu_j-f_i(u_i)f_j(u_j) \right)\prod_{\substack{l\in \{1,\ldots,n\} \\ l \neq i,j}} (u_l+{f}_l(u_l))\right) \right\},
\end{equation}
where $\mathbf{f}=(f_1,f_2,\ldots,f_n)$ and $\mathbf{u}= (u_1,u_2, \ldots, u_n)$.

In the imprecise setting we work in a finitely-additive probability space. Endogenous shocks are given by random variables whose distributions $F_i$ belong to $\mathcal{F}_{(\underline{F}_i, \overline{F}_i)}$. We define the generating functions using the ideas of Equations \eqref{eq-phi-ext} and \eqref{eq-chi-ext} so that, in particular, they do suffice the above defining relations. In addition, by Lemmas \ref{lem-phi-order} and \ref{lem-chi-order}, and by Equation \eqref{eq:rewrite} we deduce that
\[
    \begin{alignedat}{5}
       F_i'&\leqslant F_{i}\ \hbox{implies}\ \varphi'_i &\leqslant \varphi_i\ \mbox{and}\ f'_i &\leqslant f_i \quad \mbox{for}\ \ &i=1,\ldots,p,\quad\quad \mbox{and} \\
       F_j'&\leqslant F_{j}\ \hbox{implies}\ \chi'_j &\leqslant \chi_j\ \mbox{and}\ f_j &\leqslant f'_j \quad \mbox{for}\ &j=p+1, \ldots,n.\phantom{\mbox{and}}
    \end{alignedat}
\]
As in Sections \ref{subsec:imp:marsh} and \ref{subsec:imp:maxmin} we introduce the minimal and the maximal joint distribution functions
\[\begin{split}
     {\low H}^\sigma & =\min\{ H^\sigma(x_1,\ldots, x_n)~|~F_i\in \mathcal{F}_{(\underline{F}_i, \overline{F}_i)}~\mbox{for}~i=1,2,\ldots, n\}\ \mbox{and} \\
     {\up H}^\sigma  & =\max\{ H^\sigma(x_1,\ldots, x_n)~|~F_i\in \mathcal{F}_{(\underline{F}_i, \overline{F}_i)}~\mbox{for}~i=1,2,\ldots, n\}.
  \end{split}
\]

Given two vectors of functions $\underline{\mathbf{f}}\leqslant \overline{\mathbf{f}}$ such that each of their components satisfies Conditions (G1)--(G3), we let
\[
    \rmmsetcop
    =\{\rmmcop_{\mathbf{f}}\colon \underline{\mathbf{f}}\leqslant \mathbf{f} \leqslant \overline{\mathbf{f}}\},
\]
where each component of $\mathbf{f}$ also satisfies Conditions (G1)--(G3), and call this set of copulas an \emph{$n$-variate imprecise reflected maxmin (RMM for short) copula}. In Condition (G3) we apply the respective definitions of $f^*$ and $g^*$ given in Subsection \ref{subsec:reflected} in an obvious way to introduce $f_i^*$ and $f^*_j$ for $i=1,\ldots,p,$ and for $j=p+1, \ldots,n$.

We continue to use the ideas of Subsection \ref{subsec:reflected} by letting $\low f_i$ respectively $\up{f}_i$ be the minimal respectively the maximal function satisfying Conditions (G1)--(G3) and
\[
    \begin{alignedat}{3}
       \low f_i(\low G_i) &=\low F_i\widehat{F}_Z \  \mbox{if}\  \low G_i>0 \ \mbox{resp.}\ \up f_i(\up G_i) &=\up F_i\widehat{F}_Z \  \mbox{if}\  \overline{G}_i>0\ \mbox{for}\ i&=1,\ldots,p,\quad \mbox{and} \\
       \low f_j(\widehat{\up G}_j) &=\widehat{\up F}_jF_Z \  \mbox{if}\ \up G_j<1\ \mbox{resp.}\ \up f_j(\widehat{\low G}_j) &=\widehat{\low F}_jF_Z \  \mbox{if}\  \underline{G}_j<1\ \mbox{for}\ j&=p+1,\ldots,n.
    \end{alignedat}
\]
Let us summarize.
\vskip2mm

\begin{theorem}[Properties of multivariate imprecise RMM copulas]\label{thm:reflected multi}\label{thm:12}
	In the situation described above we have:
  \begin{enumerate}[(i)]
    \item $\underline{\mathbf{f}} \leqslant \overline{\mathbf{f}}$.
    \item For $\underline{\rmmcop_{{\mathbf{f}}}}=\inf\rmmsetcop$ and $\overline{\rmmcop_{{\mathbf{f}}}}=\sup\rmmsetcop$ we have that
    \[
    \begin{alignedat}{2}
    \underline{\rmmcop_{{\mathbf{f}}}}&=\bigwedge_{\widetilde{f}_i\in\{\low f_i,\up f_i\}}\rmmcop_{{\widetilde{\mathbf{f}}}}&=\bigwedge_{\substack{i\in \{1,\ldots,p\}  \\ j \in \{p+1,\ldots,n\}\\{\widetilde{f}_i=\up f_i,\widetilde{f}_j=\up f_j, \widetilde{f}_l=\low f_l,l\neq i,j}}} \rmmcop_{{\widetilde{\mathbf{f}}}}\\
    &\quad\quad\quad\quad\quad\mbox{and}\\
    \overline{\rmmcop_{{\mathbf{f}}}}&=\bigvee_{\widetilde{f}_i\in\{\low f_i,\up f_i\}}\rmmcop_{{\widetilde{\mathbf{f}}}}&=\bigvee_{\substack{i\in \{1,\ldots,p\}  \\ j \in \{p+1,\ldots,n\}\\{\widetilde{f}_i=\low f_i, \widetilde{f}_j=\low f_j, \widetilde{f}_l=\up f_l,l\neq i,j}}} \rmmcop_{{\widetilde{\mathbf{f}}}},
    \end{alignedat}
    \]
    where the infimum and supremum are attained pointwise.
    \item \[
    \begin{alignedat}{3}
       G_i & =F_iF_Z,\quad \underline{G}_i&=\underline{F}_iF_Z,\quad \overline{G}_i&=\overline{F}_iF_Z,\quad\mbox{for}\quad i=1,\ldots,p,\ \ \mbox{and} \\
       \widehat{G}_j & =\widehat{F}_j\widehat{F}_Z,\quad \widehat{\low G}_j & =\widehat{\low F}_j\widehat{F}_Z,\quad \widehat{\up G}_j & =\widehat{\up F}_j \widehat{F}_Z,\quad\mbox{for}\quad j=p+1,\ldots,n.
    \end{alignedat}
    \]
    \item \[
    \begin{split}
       \low G_i\leqslant G_i\leqslant \up G_i & \quad\mbox{for}\quad i=1,\ldots,p,\ \ \mbox{and} \\
       \widehat{\low G}_j\geqslant \widehat{G}_j\geqslant \widehat{\up G}_j & \quad\mbox{for}\quad j=p+1,\ldots,n.
    \end{split}
    \]
    \item \[
    \begin{alignedat}{4}
        \low f_i(\low G_i) &=\low F_i-\low G_i, \ \up f_i(\up G_i) &=\up F_i -\up G_i, \ \ \mbox{if}\ \low G_i,\up G_i>0\ \mbox{resp.}\ \mbox{for}\ i&=1,\ldots,p,\ \mbox{and} \\
       \low f_j(\widehat{\up G}_j) &=\widehat{\up F}_j-\widehat{\up G}_j, \ \up f_j(\widehat{\low G}_j)\ &=\widehat{\low F}_j-\widehat{\low G}_j, \ \mbox{if}\ \low G_j,\up G_j<1\ \mbox{resp.}\ \mbox{for}\ j&=p+1,\ldots,n.
    \end{alignedat}
    \]
    \item $\low f_i^*(\low G_i)\up f_j^*(\widehat{\low G}_j)=1,\quad$ $\low f_i^*(\low G_i)\low f_j^*(\widehat{\up G}_j)=1,\quad$ $\up f_i^*(\up G_i)\low f_j^* (\widehat{\up G}_j)=1,\quad$ $\up f_i^*(\up G_i)\up f_j^*(\widehat{\low G}_j)=1,$\quad for $i=1,\ldots,p$ and $j=p+1,\ldots,n$.
    \item \[
        H^\sigma(x_1, \ldots,x_n)=\rmmcop_\mathbf{f}(G_1(x_1),\ldots, G_p(x_p),\widehat{G}_{p+1}(x_{p+1}),\ldots,\widehat{G}_n(x_n)).
        \]
    \item \[\begin{split}
                 & \rmmcop_{\low{\mathbf{f}}}(\low{G}_1(x_1),\cdots, \low{G}_p(x_p), \widehat{\up G}_{p+1}(x_{p+1}) ,\ldots, \widehat{\up G}_n(x_n)) \\
               \leqslant\  & \rmmcop_{{\mathbf{f}}} (G_1(x_1),\ldots,G_p(x_p), \widehat{{G}}_{p+1}(x_{p+1}), \ldots,\widehat{{G} }_n(x_n)) \\
               \leqslant\  & \rmmcop_{\up{\mathbf{f}}} (\up G_1(x_1),\ldots,\up G_p(x_p), \widehat{\low{G}}_{p+1}(x_{p+1}), \ldots,\widehat{\low{G} }_n(x_n)).
            \end{split}
    \]
    \item \[ \begin{split}
                \low H^\sigma(x_1, \ldots,x_n) & =\rmmcop_{\low{\mathbf{f}}} (\low G_1(x_1),\ldots, \low G_p(x_p), \widehat{\up G}_{p+1}(x_{p+1}) ,\ldots, \widehat{\up G}_n(x_n)),                 \\
                \up H^\sigma(x_1, \ldots,x_n) & =\rmmcop_{\up{\mathbf{f}}} (\up G_1(x_1),\ldots,\up G_p(x_p), \widehat{\low{G}}_{p+1}(x_{p+1}), \ldots,\widehat{\low{G} }_n(x_n)).
             \end{split}
        \]
  \end{enumerate}
\end{theorem}

\begin{proof}
  \emph{(i)} follows by the observations above, \emph{(ii)} follows from Equation \eqref{eq:rmm:2} after a short computation. \emph{(iii)}, \emph{(iv)} and \emph{(v)} are immediate from the above.
  \emph{(vi)} Let us compute only one of the four cases that go in a similar way:
  \[
    \low f_i^*(\low G_i)=\frac{\low f_i(\low G_i)}{\low G_i}=\frac{\low F_i-\low G_i}{\low G_i}=\frac{1}{F_Z}-1=\frac{\widehat{F}_Z}{F_Z},\ \mbox{for}\ i=1,\ldots,p,
  \]
  and
  \[
    \up f_j^*(\widehat{\low G}_j)=\frac{\up f_j(\widehat{\low G}_j)}{\widehat{\low G}_j}=\frac{\widehat{\low F}_j-\widehat{\low G}_j}{\widehat{\low G}_j}= \frac{1}{\widehat{F}_Z}-1=\frac{{F}_Z}{\widehat{F}_Z},\ \mbox{for}\ j=p+1,\ldots,n.
  \]
  Point \emph{(vii)} amounts to the same as Equation \eqref{eq:rmm:1}. To get \emph{(viii)} apply Equation \eqref{eq:rmm:2} to the formula $C= \rmmcop_{\low{\mathbf{f}}} (\low{G}_1(x_1),\cdots, \low{G}_p(x_p), \widehat{\up G}_{p+1}(x_{p+1}) ,\ldots, \widehat{\up G}_n(x_n))$ and write the expression under the min operator as a product of three factors
  \[
    A_{ij}=\low{G}_i(x_i)\widehat{\up G}_j(x_j)-\underline{f}_i (\low{G}_i(x_i)) \underline{f}_j(\widehat{\up G}_j(x_j)) \]
  \[
    B_i=\prod_{\substack{l=1 \\ l \neq i}}^p (\low{G}_l(x_l)+\low{f}_l(\low{G}_l(x_l))) \quad\mbox{and}\quad D_j=\prod_{\substack{l=p+1 \\ l \neq j}}^n (\widehat{\up G}_l(x_l)+\low{f}_l(\widehat{\up G}_l(x_l))),
  \]
  for $i=1,\ldots,p,$ and $j=p+1,\ldots,n$.
  A straightforward computation using \emph{(v)} yields $A_{ij}=
  \underline{F}_i(x_i)\widehat{\overline{F}}_j(x_j)(\widehat{F}_Z(x_j) -\widehat{F}_Z(x_i))$ so that
\[
  A_{ij}\leqslant {F}_i(x_i)\widehat{F}_j(x_j)(\widehat{F}_Z(x_j)-\widehat{F}_Z(x_i))
  = {G}_i(x_i)\widehat{G}_j(x_j)-f_i({G}_i(x_i)) f_j(\widehat{ G}_j(x_j)).
\]
In a similar way we get
\[
    B_i\leqslant\prod_{\substack{l=1 \\ l \neq i}}^p ({G}_l(x_l)+{f}_l({G}_l(x_l))) \quad\mbox{and}\quad D_j\leqslant\prod_{\substack{l=p+1 \\ l \neq j}}^n (\widehat{ G}_l(x_l)+ {f}_l(\widehat{G}_l(x_l))),
\]
and the first desired inequality follows. Considerations of the same kind yield the second one. Finally, Point \emph{(ix)} follows from Points \emph{(vii)} and \emph{(viii)}.
\end{proof}

\textbf{Remark.} Note that in Theorem \ref{thm:11}, there is no statement that is equivalent to Point \emph{(ii)} of Theorem \ref{thm:12}, because it appears that no such statement can be proven for multivariate imprecise MM copulas. Thus in a sense, the multivariate imprecise RMM copulas behave nicer from the point of view of pointwise order than multivariate imprecise MM copulas.

\section{Conclusion}

The uncertainty of the final outcome of the rules of modeling dependencies in the bivariate imprecise setting issues a warning that one should address this issue on the multivariate imprecise level with utmost caution. A view on this problem was presented in the last but one paragraph of our introduction. 
Therefore, the paper \cite{OmSk} was helpful to the specialists in the area giving two important examples of bivariate imprecise copulas, Marshall's and maxmin copulas; there, the background assumptions of the shock model inducing each of the two families of copulas were assumed imprecise thus leading to a naturally defined imprecise copulas that have many interesting additional properties including coherence.

There is an important fact, namely \cite[Theorem 4]{OmSt2}, saying that the copulas obtained via the Sklar's theorem from bivariate distributions on finitely additive probability spaces are the same as the ones obtained on the standard probability spaces. This means that whenever the controversy of the bivariate imprecise dependence is resolved, it will be resolved both for the standard and for the non-standard approach simultaneously.

This encourages us to present an investigation of three major multivariate cases of shock model induced copulas: the Marshall's, the maxmin, and the reflected maxmin copulas (RMM). We believe that the properties of these objects, no matter what they will be called in the end, will help further investigations in the area. Here are some quick findings of ours. In all the three cases the extreme values of the generators lead us to extreme values of the set of copulas and to extreme values of the set of distributions.
In the case of Marshall's copulas this correspondence is simple and expected, a direct extension of the bivariate case. The set of copulas (a possible candidate for an $n$-variate imprecise copula) is coherent (Theorem \ref{thm:marshall multi}\emph{(ii)}), the set of the corresponding joint distributions is coherent (Theorem \ref{thm:marshall multi}\emph{(v)}\&\emph{(vi)}) and the lower and upper bounds correspond to each other.

The maxmin copulas behave somewhat differently. The set of joint distributions is coherent (Theorem \ref{thm:11}\emph{(vii)}\&\emph{(viii)}); however, the question of coherence of the (possible candidate for an) $n$-variate imprecise copula is left open. In the bivariate case one is able to make this set coherent as well, although the extremes were not correspondent to the according extreme distributions. The case of RMM copulas is more involved but in some sense clearer. The set of joint distributions is coherent\footnote{Here we understand the term \emph{coherent} in the usual way, namely, it means that the value of the lower respectively upper bound can be approximated at any fixed point of the unit square by values of the copulas from the set; actually in our case this value is even attained.} (Theorem \ref{thm:12}\emph{(viii)}\&\emph{(ix)}) and we can express the lower and the upper bound of the (possible candidate for an) $n$-variate imprecise copula as a minimum, respectively maximum of a finite number of copulas that belong to a specific set (Theorem \ref{thm:12}\emph{(ii)}). The obtained bounds are quasi-copulas in general and the solution to the question of coherence of this set needs methods that are yet to be discovered.
(For the bivariate case the kind of methods were developed in \cite{OmSt1}.) Since the maxmin and RMM copulas are obtainable from each other through a number of reflections, it is possible that a similar conclusion as the one exhibited in Theorem \ref{thm:12}\emph{(ii)} exists for maxmin copulas as well, but in view of the last remark in the paper, this looks like a nontrivial task for further investigations.
\\

Consequently, our paper opens a number of questions left to the community of experts on imprecise copulas to solve. These are primarily tasks in the multivariate imprecise setting:
\begin{enumerate}
  \item How to define a $p$-box of multivariate joint distributions?
  \item What to adopt as an imprecise multivariate copula?
  \item One needs a Sklar type theorem connecting the two notions above.
  \item One needs to develop an $n$-variate coherence testing algorithm for a set of (quasi)copulas extended from the bivariate case (cf.\ \cite{OmSt1}).
  \item One needs to develop an $n$-variate coherence testing algorithm for a set of (quasi)distributions extended from the bivariate case (cf.\ \cite{OmSt2}).
\end{enumerate}


\begin{thebibliography}{00}



\bibitem{AuCoCoTr} T.\ Augustin, F.\ P.\ A.\ Coolen, G. de Cooman, M.\ C.\ M.\ Troffaes (editors), \textsl{Introduction to imprecise probabilities}. John Wiley \& sons, Chichester (2014).


\bibitem{Cool} F.\ P.\ A.\ Coolen, \textsl{On the Use of Imprecise Probabilities in Reliability}, Quality and Reliability Engineering International. \textbf{20} (2004), 193--202.

\bibitem{Couso2000} I.\ Couso, S.\ Moral, P.\ Walley, \textsl{A survey of concepts of independence for imprecise probabilities}, Risk, Decision and Policy, \textbf{5} (2000), 165--181.

\bibitem{Couso2010} I.\ Couso, S.\ Moral, \textsl{Independence concepts in evidence theory}, International Journal of Approximate Reasoning, \textbf{51} (2010), 748--758.


\bibitem{DeCoHeQu}
G. de~Cooman, F. Hermans, E. Quaeghebeur,
\newblock Imprecise {M}arkov chains and their limit behavior.
\newblock \textsl{Probability in the Engineering and Informational Sciences}, {\bf
  23.4}(2009), 597--635.

\bibitem{DuSe} F.~Durante, C.~Sempi. {\sl Principles of Copula Theory}. CRC/Chapman \& Hall, Boca Raton (2015).

\bibitem{DuGiMa} F.~Durante, S.~Girard, G.~Mazo, \textsl{Marshall–Olkin type copulas generated by a global shock}, Journal of Computational and Applied Mathematics, \textbf{296} (2016), 638--648.


\bibitem{DuOmOrRu} F.~Durante, M.\ Omladi\v{c}, L.\ Ora\v{z}em, N.\ Ru\v{z}i\'{c}, \textsl{Shock models with dependence and asymmetric linkages}, Fuzzy Sets and Systems, \textbf{323} (2017), 152--168.

\bibitem{FeKrGiMySe} S.\ Ferson, V.\ Kreinovich, L.\ Ginzburg, D.\ S.\ Myers, K.\ Sentz, \textsl{Constructing probability boxes and Dempster-Shafer structures}. Technical report SAND2002-4015, 2003.

\bibitem{JaScAu} C.\ Jansen, G.\ Schollmeyer, T.\ Augustin, \textsl{Concepts for decision making under severe uncertainty with partial ordinal and partial cardinal preferences}, International Journal of Approximate Reasoning, \textbf{98} (2018), 112--131.

\bibitem{Joe} H.\ Joe. {\sl Dependence Modeling with Copulas.} Chapman \& Hall/CRC, London, 2014.


\bibitem{KoOm} T.\ Ko\v{s}ir, M.\ Omladi\v{c}, \textsl{Reflected maxmin copulas and modeling quadrant subindependence}, Fuzzy Sets and Systems \textbf{378} (2020), 125--143.



\bibitem{KoBuKoMoOm2} D.\ Kokol Bukov\v{s}ek, T.\ Ko\v{s}ir, B.\ Moj\v{s}kerc, M.\ Omladi\v{c}, \textsl{Asymmetric linkages: maxmin vs.\ reflected maxmin copulas}, Fuzzy sets and systems, \textbf{393} (2020), 75--95. 

\bibitem{MaOl} A.~W.~Marshall, I.~Olkin, \textsl{A multivariate exzponential distributions}, J.~Amer.~Stat.~Assoc., \textbf{62}, (1967), 30--44.

\bibitem{Mars} A.~W.~Marshall, \textsl{Copulas, marginals, and joint distributions}, in: L.~R\"{u}schendorf, B.~Schweitzer, M.~D.~Taylor  (eds.), Distributions with Fixed Marginals and Related Topics in LMS, Lecture Notes -- Monograph Series, \textbf{28} (1996), 213--222.

\bibitem{MiMo} E.\ Miranda, I.\ Montes, \textsl{Shapley and Banzhaf values as probability transformations}, International Journal of Uncertainty, Fuzziness and Knowledge-Based Systems. \textbf{26} (2018), 917--947.

\bibitem{MoMiMo} I.\ Montes, E.\ Miranda, S.\ Montes, \textsl{Decision making with imprecise probabilities and utilities by means of statistical preference and stochastic dominance}. European Journal of Operational Research, \textbf{234} (2014), 209--220.

\bibitem{MoMiPeVi} I.\ Montes, E.\ Miranda, R.\ Pelessoni, P,\ Vicig, \textsl{Sklar's theorem in an imprecise setting}, Fuzzy Sets and Systems, \textbf{278} (2015), 48--66.

\bibitem{Nau2011} R.\ Nau, \textsl{Imprecise probabilities in Non-cooperative games}. In Proceedings of ISIPTA (2011), 297--306.

\bibitem{Nels} R.\ B.\ Nelsen. {\sl An introduction to copulas}. 2nd edition, Springer-Verlag, New York (2006).

\bibitem{ObKiSc} M.\ Oberguggenberger, J.\ King, B.\ Schmelzer, \textsl{Classical and imprecise probability methods for sensitivity analysis in engineering: a case study}, International Journal of Approximate Reasoning, \textbf{50} (2009), 680--693.

\bibitem{OmRu}  M.\ Omladi\v{c}, N.\ Ru\v{z}i\'{c}, \textsl{Shock models with recovery option via the maxmin copulas}, Fuzzy Sets and Systems, \textbf{284} (2016), 113--128.

\bibitem{OmSt1} M.\ Omladi\v{c}, N.\ Stopar, \textsl{Final solution to the problem of relating a true copula to an imprecise copula}, Fuzzy sets and systems, \textbf{393} (2020), 96--112. 

\bibitem{OmSt2} M.\ Omladi\v{c}, N.\ Stopar, \textsl{A full scale Sklar's theorem in the imprecise setting}, Fuzzy sets and systems, \textbf{393} (2020), 113--125. 

\bibitem{OmSk} M.\ Omladi\v{c}, D. \v{S}kulj, \textsl{Constructing copulas from shock models with imprecise distributions}, International Journal of Approximate Reasoning, \textbf{118} (2020), 27--46

\bibitem{PeVi} R.\ Pelessoni, P.\ Vicig, \textsl{Convex Imprecise Previsions}, Reliable Computing, \textbf{9} (2003), 465--485.

\bibitem{PeViMoMi1} R. Pelessoni, P. Vicig, I. Montes, and E. Miranda, \textsl{Imprecise copulas and bivariate stochastic orders}. In: Proc. EUROFUSE 2013, Oviedo 2013, 217--224.

\bibitem{PeViMoMi2} R.\ Pelessoni, P.\ Vicig, I.\ Montes, E.\ Miranda, \textsl{Bivariate $p$-boxes}, {International Journal of Uncertainty, Fuzziness and Knowledge-Based Systems}, {\bf 24.02} (2016), 229--263.

\bibitem{Schmelzer2015} B.\ Schmelzer, \textsl{Joint distributions of random sets and their relation to copulas}, International Journal of Approximate Reasoning, \textbf{65} (2015), 59--69.

\bibitem{Schmelzer2015b} B.\ Schmelzer, \textsl{Sklar's theorem for minitive belief functions}, International Journal of Approximate Reasoning, \textbf{63} (2015), 48--61.

\bibitem{Schmelzer2018} B.\ Schmelzer, \textsl{Multivariate capacity functional vs. capacity functionals on product spaces}. Fuzzy Sets and Systems, \textbf{364} (2019), 1--35.



\bibitem{Skla} A.\ Sklar, \textsl{Fonctions de r\'{e}partition \`{a} $n$ dimensions et leurs marges}, Publ.\ Inst.\ Stat.\ Univ.\ Paris \textbf{8} (1959), 229--231.

\bibitem{skul}
D.\ \v{S}kulj,
\newblock \textsl{Discrete time {M}arkov chains with interval probabilities.}
\newblock {International Journal of Approximate Reasoning}, {\bf 50.9} (2009),
  1314--1329.

\bibitem{Troffaes2007} M.\ C.\ M.\ Troffaes, \textsl{Decision making under uncertainty using imprecise probabilities}, International Journal of Approximate Reasoning, \textbf{45} (2007), 17--29.


\bibitem{TrDe} M.\ C.\ M.\ Troffaes, S. Destercke \textsl{Probability boxes on totally preordered spaces for multivariate modelling}, International Journal of Approximate Reasoning, \textbf{52} (2011), 767--791.

\bibitem{UtCo} L.\ V.\ Utkin, F.\ P.\ A.\ Coolen, \textsl{Imprecise Reliability: An Introductory Overview.} In: Levitin G. (eds) Computational Intelligence in Reliability Engineering. Studies in Computational Intelligence, \textbf{40} (2007). Springer, Berlin, Heidelberg.

\bibitem{Vicig2008} P.\ Vicig, \textsl{Financial risk measurement with imprecise probabilities}, International Journal of Approximate Reasoning, \textbf{49} (2008), 159--174.

\bibitem{Wall} P. \ Walley, \textsl{Statistical Reasoning with Imprecise Probabilities.} Chapman and Hall, London, 1991.

\bibitem{YuDeSaSc} L.\ Yu, S.\ Destercke, M.\ Sallak, W.\ Schon, \textsl{Comparing system reliability with ill-known probabilities}, Proceedings of IPMU 2016, 619--629.





\end{thebibliography}
\end{document}